\documentclass[12pt]{amsart}
\usepackage{amsmath,amssymb,amsthm,amscd}
\usepackage[text={160mm,230mm},centering,headsep=2mm,footskip=5mm]{geometry}
\usepackage[dvipsnames]{xcolor}
\usepackage[colorlinks=true,linkcolor=RoyalBlue,citecolor=Mahogany,pagebackref]{hyperref}
\usepackage{graphicx}
\usepackage{tikz}
\usetikzlibrary{patterns}
\usepackage[all]{xy}
\SelectTips{cm}{10}
\everyxy={<2.5em,0em>:} 

\newcounter{CountAlpha}

\theoremstyle{theorem}
\newtheorem{MainThm}[CountAlpha]{Theorem}
\newtheorem{Thm}{Theorem}[section]

\newtheorem{Cor}[Thm]{Corollary}
\newtheorem{Prop}[Thm]{Proposition}

\theoremstyle{definition}
\newtheorem{Def}[Thm]{Definition}
\newtheorem{Ex}[Thm]{Example}
\newtheorem{Rk}[Thm]{Remark}

\newtheorem{Obs}[Thm]{Observation}

\newcommand{\cone}{\operatorname{cone}}

\newcommand{\Spec}{\operatorname{Spec}}
\newcommand{\Hom}{\operatorname{Hom}}

\newcommand{\gqz}{{\geq 0}}
\newcommand{\NP}{\Delta} 
\newcommand{\Bl}{\mathrm{B\ell}}

\newcommand{\IA}{\mathbb{A}}
\newcommand{\IF}{\mathbb{F}}
\newcommand{\IG}{\mathbb{G}}
\newcommand{\IR}{\mathbb{R}}
\newcommand{\IQ}{\mathbb{Q}}
\newcommand{\IZ}{\mathbb{Z}}

\newcommand{\cA}{\mathcal{A}}

\newcommand{\cE}{\mathcal{E}}
\newcommand{\cF}{\mathcal{F}}
\newcommand{\cI}{\mathcal{I}}

\newcommand{\cO}{\mathcal{O}}
\newcommand{\cS}{\mathcal{S}}
\newcommand{\cX}{\mathcal{X}}

\newcommand{\pure}{{\operatorname{pure}}}
\newcommand{\rd}{{\operatorname{red}}}

\definecolor{gruen}{rgb}{0, 0.55, 0}

\numberwithin{equation}{section}

\newcommand{\double}{\genfrac..{0pt}1
{\raise -1pt\hbox{$\scriptstyle\longrightarrow$}}{\raise 3pt\hbox
{$\scriptstyle\longrightarrow$}}}

\begin{document}

\title{Desingularization of binomial varieties using toric Artin stacks}

\makeatletter
\@namedef{subjclassname@2020}{\textup{2020} Mathematics Subject Classification}
\makeatother


\author{Dan Abramovich}
\address{\tiny Department of Mathematics, Brown University, Box 1917,
	Providence, RI~02912, USA}
\email{dan\_abramovich@brown.edu}
\thanks{Research by D.A. was supported in part by
funds from NSF grant DMS-2100548 and a Simons fellowship.  He thanks ICERM and its Combinatorial
Algebraic Geometry program for its hospitality during Spring 2021, Institut Mittag-Leffler and its Moduli and Algebraic Cycles program (Swedish
Research Council grant no. 2016-06596) where he was Visiting Professor
during September--October 2021, and the Einstein Institute of Mathematics at
Jerusalem for its hospitality during November--December 2021.}

\author{Bernd Schober}
\address{\tiny Carl von Ossietzky Universit\"at Oldenburg, Institut f\"ur Mathematik,
	Ammerl\"ander Heerstra{\ss}e 114 - 118,	26129 Oldenburg (Oldb), Germany}
\curraddr{\tiny None (Hamburg, Germany).}
\email{schober.math@gmail.com}
\thanks{Research by B.S.~was supported in part by the DFG-project ``Order zeta functions and resolutions of singularities" (DFG project number: 373111162).}

\date{\today}	

\begin{abstract}
	We show how the notion of fantastacks can be used to effectively desingularize binomial varieties defined over algebraically closed fields.
	In contrast to a desingularization via blow-ups in smooth centers,
	we drastically reduce the number of steps and the number of charts appearing along the process. 
	Furthermore, we discuss how our considerations extend to a partial simultaneous normal crossings desingularization of finitely many binomial hypersurfaces. 
\end{abstract}

\maketitle

\setcounter{tocdepth}{1}
\tableofcontents

\section{Introduction}

	Motivated by recent  works with focus on stacks in the service of resolution
	of singularities over fields of characteristic zero 
	\cite{AQ,ATW_weighted,ATW_jems,ATW_destack,Very_fast,MingHao},
	we investigate the case of binomials defined over algebraically closed fields of arbitrary characteristic from the perspective of stacks, and with a view towards efficient computations of resolutions and of $p$-adic and motivic integrals. Classical resolution approaches are discussed in the recent \cite{Gaube}.
	
	We work over an algebraically closed field $K$ so as to keep technicalities at a minimum.
	
	\subsection{Key ingredient: fantastacks} Building on the results of \cite{GPT,GP,Teissier},
	the key new ingredient for our approach is the notion of fantastacks by Geraschenko and Satriano \cite{Toric1},
	which are examples of toric Artin stacks, and 
	 a birational reinterpretation of the toric Cox construction.
	Up to isomorphism,
	a fantastack $ \cF_{\Sigma, \beta} $ can be constructed from a fan $\Sigma$ on the lattice {$N\simeq  \IZ^m $ and
	a homomorphism of lattices $ \beta \colon \IZ^n \to N $ with certain properties. 
	For a more detailed discussion of fantastacks 
	we refer to Section~\ref{Sec:BasicsFanta}.

	\subsection{Setup: binomials in toric varieties}\label{Sec:setup-binomials}
	
	A binomial in affine space $\Spec K[x_1,\ldots,x_{m}]$ always has the form 
	\begin{equation} \label{Eq:binomial} f \quad = \quad x_1^{C_1}\cdots x_{{m}}^{C_{{m}}}\ \ \left(x_1^{A_1}\cdots x_r^{A_r}\ \  - \ \ \lambda y_1^{B_1} \cdots y_s^{B_s}\right),\end{equation}
	where $\{x_1,\ldots,x_r\}$ and $\{y_1,\ldots,y_s\}$ are disjoint subsets of $\{x_1,\ldots x_{{m}}\}$ and $\lambda \in K^\times$.  The \emph{monomial factor} $x_1^{C_1}\cdots x_{{m}}^{C_{{m}}}$ will be of secondary interest and the \emph{purely binomial factor} $x_1^{A_1}\cdots x_r^{A_r} - \lambda y_1^{B_1} \cdots y_s^{B_s}$ is the main factor to be addressed.
	
	Note that the purely binomial factor is nonreduced if and only if $ A_i, B_i$ are all divisible by the characteristic of $K$, as $K$ is assumed algebraically closed. 
	
	In Section~\ref{Sec:binomials-general} we extend the notion of binomials from affine space to the generality of smooth toric varieties and stacks.
	
	Tevelev \cite[Definition 1.3]{Tevelev} defined \emph{sch\"on subvarieties} of toric varieties, which we do not reproduce here, as our special situation is simpler. 
	In Section~\ref{Sec:binomials-schon} we specialize to the case of binomials, in the generality of toric stacks. 
	Our usage is slightly relaxed since we allow monomial factors: 
	in affine charts, a binomial \eqref{Eq:binomial}} is \emph{sch\"on}  if either all $A_i=0$ or all $B_j=0$.
	
\subsection{Arrangements of binomials and binomial subschemes}	
	
Given a finite collection of binomials, locally represented by $f_1,\ldots, f_a$, they together define an arrangement of subschemes as the zero-set $X := V(f_1 \cdots f_a)$ of their product. The arrangement is said to be sch\"on if each $f_i$ is sch\"on.

	If all $f_i$ are sch\"on pure binomials, the binomial ideal $\cI= \langle f_1,\ldots,f_a \rangle $ they generate defines a sch\"on subscheme  in the sense of \cite[Definition 1.3]{Tevelev}. The situation is a bit more involved when the monomial factors are not trivial, see Section \ref{binomial ideal}.

\subsection{Simple arrangements and problematic primes}

In Section~\ref{Sec:simple-arrangement} we recall the notion of a simple arrangement of smooth subschemes, and upgrade it to stacks.

	In characteristic 0, a sch\"on arrangement of binomials on a toric stack is automatically a simple arrangement. This is not the case in positive characteristic, a phenomenon also described in Section~\ref{Sec:simple-arrangement}, Observation~\ref{Obs:problematic} --- there is a finite set of prime characteristics $\cE$, computable from the collection of exponents of the binomials $f_i$, outside of which the binomials form a simple arrangement.   

\subsection{Subdivisions and modifications associated to binomials}\label{Sec:GPT}
Consider again a binomial \eqref{Eq:binomial} in affine space. The Newton polyhedron of $f$ defines a singular subdivision $\Sigma_f$ of the fan $\Sigma_0$ of affine space, and a corresponding singular toric modification of affine space. It is shown in \cite{GPT}, \cite[Section 3]{GP}, \cite[Section 6]{Teissier} that any smooth subdivision $\Sigma$ of $\Sigma_f$ results in a resolution of $\{f=0\}$; however such a smooth subdivision is typically computationally expensive. Our goal here is precisely to avoid such subdivision.

Similarly, for a collection $f_1,\ldots,f_a {\in K[x_1, \ldots, x_m]} $ of binomials we obtain a subdivision of $\Sigma_0$ into a fan $\Sigma_f$ of at most $2^a$ cones with $<\ \binom{{m}+a}{{m}-1}$ edges, independently of the exponents and coefficients of $f_i$  --- it is the subdivision dual to the Newton polyhedron of the product $f_1 \cdots f_a$. By the above references  any smooth subdivision of $\Sigma_f$ resolves all the $f_j$. Here we provide a stack-theoretic resolution requiring no subdivision of $\Sigma_f$ at all.

\subsection{Main result on arrangements of binomials}
Our main point is that the fantastack associated to the fan $\Sigma_f$, or any subdivision $\Sigma$ thereof, gives such a simultaneous   resolution of $f_i$  immediately.

	\begin{MainThm}[See Theorem~\ref{Thm:1Text}]
		\label{Thm:Main}
		Let $ K $ be an algebraically closed field of arbitrary characteristic $ p \geq 0 $. 
		Let $ f_1, \ldots, f_a \in K[x] = K [x_1, \ldots, x_m] $ be finitely many binomials,
		where $ a, m  \in \IZ_+ $ with $ m \geq 2 $.
		Let $ \mathcal{E} \subset \IZ $ be {the} set of problematic primes associated to the exponents of the pure binomial factors of $f_1, \ldots, f_a $.

		Let $ \Sigma $ be any subdivision of the dual fan $\Sigma_f$ of the Newton polyhedron of the product $ f_1 \cdots f_a $, 
		 inducing the morphism of smooth toric stacks $ \rho \colon  \cF_{\Sigma,\beta} \to \IA^m $, and let $ X := V(f_1 \cdots f_a ) \subset \IA^m $.

			Then the reduced preimage $ \rho^{-1}(X)_{\rm red} \subset \cF_{\Sigma, \beta} $ is a sch\"on arrangement of smooth binomials.

			If furthermore $ p \notin \mathcal{E} $,
			then
			$ \rho^{-1}(X)_{\rm red}  $ induces a simple arrangement on $  \cF_{\Sigma, \beta} $.
			
	\end{MainThm}

\subsection{Remarks}
\phantomsection\label{Rem:mainA}
\begin{enumerate}
\item 
The main reason for our restriction to algebraically closed fields is that, as far as we know, there exists no detailed reference on the theory of fantastacks over an arbitrary base. Note that the treatment of \cite{Toric1} applies in any characteristic, though the classification result, not used here, is only stated in characteristic 0.%
	\footnote{According to \cite[Remark~1.2]{Toric1}, their restriction to algebraically closed fields is made in order to avoid technicalities.
	}

\item 	For $ a = 1 $, 
	we obtain a desingularization of the binomial hypersurface defined by $ f_1 $,
	while for $ a = 2 $, assuming $p\notin \cE$, Theorem~\ref{Thm:Main} provides us with a simultaneous normal crossings desingularization of the binomial hypersurfaces given by $ f_1 $ and $ f_2 $.

\item 	Notice that the ambient space in Theorem \ref{Thm:Main} is an Artin stack and not a scheme. However, any triangulation of its fan replaces it by a Deligne--Mumford stack; this is achieved, without adding rays, using star subdivisions $\Sigma' \to \Sigma$ at all the rays of $\Sigma$. The resulting stack $\cF_{\Sigma',\beta}$ is a Deligne--Mumford open substack of $\cF_{\Sigma,\beta}$, and requires no additional computation. 
	
	We also note that any smooth subdivision provides a scheme where the result holds; this recovers earlier known results, such as those of \cite{GPT}, but also loses the efficiency of the current method.

\item 	In general, when $ a \geq 3 $ and $p\notin \cE$,
	Theorem~\ref{Thm:Main} does not give a simultaneous normal crossings  desingularization, 
	as the resulting arrangement of hypersurfaces is not in normal crossings position in general,
	see Example~\ref{Ex:not_all}. This  scenario is present also in  earlier work which uses smooth blowup centers. To resolve a simple arrangement and make it normal crossings one may use well--known procedures requiring at most $m$ blowups, see \cite{Hu,Li}.
	We recall this in  
	Theorem~\ref{Thm:Hu}. 
	
	The situation where $p\in \cE$ is quite interesting. One expects  to have an elegant procedure for a simultaneous normal crossings desingularization. We do not pursue this question here.

\item 
	While we do not give a detailed complexity analysis, we note that the geometric complexity encoded in the number of cones of the fan, in particular  the number of charts and divisors, is bounded solely in terms of the ambient dimension and the number $a$ of binomials, and not depending on the exponent or coefficients of the generators. 
	For future discussion we will say that such process is of \emph{combinatorially bounded complexity}.

	\end{enumerate}

\subsection{Main result on binomial varieties} \label{Sec:mainresult}

In the context of desingularization of binomial {schemes} using fantastacks, we prove the following result. 
Following  earlier sources, an ideal $\cI \subset k[x_1,\ldots,x_m]$ is \emph{binomial} if it is generated by binomials and a subscheme $X \subset \IA^m$ is \emph{binomial} if its ideal is binomial. The subscheme is \emph{purely binomial} if it has no component or embedded component supported on the boundary divisor $V(x_1 \cdots x_m)$. 
We denote by $X^{\pure}$ the \emph{purely binomial part} of a binomial subscheme $X$, namely the closure of its intersection with the torus.

Throughout the paper we denote the reduced scheme underlying a scheme $X$ by $X_\rd$, and call it \emph{the reduction} of $X$ for brevity. 

 For describing our result we introduce the following somewhat ad-hoc terminology: a  subscheme $X$ of a smooth scheme $V$ is said to be \emph{in simple normal  position} if the components of its reduction $X_\rd$ are smooth, and at any geometric point $p$ of $V$ there are local parameters $x_i \in \hat\cO_{V,p} $ such that $X_\rd$ is defined by a monomial ideal in $x_i$. This is the same as saying that,  in these coordinates,  $X_\rd$ is the union of coordinate subspaces, of arbitrary dimensions, not necessarily meeting transversely. 
 As an example, two or more coordinate lines $ L_1, L_2,  \ldots $ in $\IA^n, n\geq 3$, are in simple normal position even though they do not meet transversely. 
 To see the latter, observe that $ L_1 \cap L_2 $ is the origin, which is of codimension $n<2(n-1)$. 
 The same phenomenon holds for coordinate subspaces of dimension $<n/2$.

\begin{MainThm}[See Theorem~\ref{Thm:2Text}]
		\label{Thm:MainIdeals}
		Let $ K $ be an algebraically closed field and $ m \in \IZ_{\geq 2} $.
		Let $ X \subset \IA^m $ be a binomial 		subscheme.
		Denote by $ \Sigma_0 $  the fan in $ \IR^m $ with unique maximal cone $ \IR_\gqz^m $. Let $f = \{f_1,\ldots,f_a\}$ be binomial generators of $\cI_X$. 
		Let $\Sigma_f$ denote again  the dual fan of the {N}ewton polyhedron of $f_1 \cdots f_a$, with resulting toric stack $\cF_{\Sigma_f,\beta}$. 
		\begin{enumerate}
		 \item {\rm (See \cite{GP, GPT, Teissier})} The proper transform of $X^\pure$ in $\cF_{\Sigma_f,\beta}$ is a sch\"on purely binomial subscheme, in particular its reduction is smooth.
		 \item 
		There exists a further subdivision $ \Sigma $ of $ \Sigma_0 $, of combinatorially bounded complexity, 
		with the following property:
		If $ \beta \colon \IZ^n \to \IZ^m $ is the homomorphism 
		determined by the matrix $ M \in \IZ^{m \times n} $,
		whose columns are
		the primitive generators for the rays of the fan $ \Sigma $,
		then
		\[
				\rho^{-1}(X)_{\rm red}  \subset  \cF_{\Sigma,\beta}
		\]
		is in simple normal position on $ \cF_{\Sigma,\beta} $,
		where we denote by $ \rho \colon  \cF_{\Sigma,\beta} \to \IA^m $ the morphism of toric stacks induced by $ \beta $.
		\end{enumerate}		

\end{MainThm}

Once again we note that in (2) the work of Li Li \cite{Li} allows for a further sequence, of at most $m$   blowups, such that the total transform of  $X_\rd$ is a simple normal crossings divisor. We recall this in Proposition \ref{Thm:Li}.
	
	\subsection{Prior work} Let us mention that there exist several results on the desingularization 
	of toric and binomial varieties, e.g., \cite{BM,BlancoI,BlancoII,GPT,KKMS}, 
	where the assumptions slightly vary compared to those stated here.
	In contrast to blow-ups in smooth centers, we need to apply the fantastack construction only once in Theorem~\ref{Thm:Main} and it is covered by at most $ 2^a $ affine charts with  $<\ \binom{{m}+a}{{m-1}}$ coordinates in total. 
	This number is 
	 significantly smaller  than that obtained using blow-ups in smooth centers once the exponents appearing in $ f_1, \ldots, f_a $ are large (see Remark~\ref{Rk:NumbCharts}). The primary reason is that we do not require a smooth subdivison of the dual fan.

	\subsection{Computational motivation}\label{Sec:p-adic-intro}
	Besides our interest in the problem on its own,
	another motivation for this project
	is coming from the explicit computation of $ p $-adic integrals, 
	whose data are given by binomials.
	For example, the following task appears in the context of computing local subring zeta functions: determine the integral 
	\[  
		\int_B |x_1|_p^{s-3} | x_2|_p^{s-2} | x_3 |_p^{s-1} d \mu ,
	\]
	where 
	\begin{itemize}
	\item $ |.|_p $ is the $ p $-adic absolute value on $ \IZ_p $; we denote the corresponding valuation  $ v_p (.) $;
	\item  $ d \mu $ is an additive Haar measure; and 
	\item $ B \subset \IZ_p^6 $ is the subset determined by the inequalities
	\[ 	
		\begin{array}{ll} 
		v_p (x_1) \leq v_p (x_4^2 - x_2 x_4), 
		&
		v_p (x_1) \leq v_p (x_4 x_5 - x_4 x_6), 
		\\
		v_p (x_2) \leq v_p (x_6^2),
		&
		v_p (x_1 x_2) \leq v_p (x_2 x_5^2 - x_4 x_6^2 ).
		\end{array} 
	\]
	\end{itemize}
	
	A possible approach to determine such an integral is to first perform blow-ups 
	so that all data appearing becomes monomial,
	and then to compute the monomial integrals in every chart.
	Evidently, it is desirable to keep the number of final charts as small as possible in this context. 
	Notice that even though we are in the $ p $-adic setting and not working over an algebraically closed field,
	the techniques discussed in this article can be applied at least for all but finitely many primes $ p $,
	where the problematic set of primes depends on the precise data given.
	For more details, we refer to Section~\ref{subsec:mixed}. 
	
	If one follows the method of Theorem~\ref{Thm:Main}, one has to be careful since the outcome is a toric Artin stack
	and some extra work is required.

	For algebraically closed fields of characteristic zero,
	Satriano and Usatine initiated an investigation for a method to study stringy Hodge numbers of a singular variety using motivic integration for Artin stacks in \cite{SU_Stringy,SU_change}.
	To address $p$-adic integration, an analogous framework for $ p $-adic integration on Artin stacks needs to be developed.
	Alternatively, one could efficiently replace the Artin stack by a Deligne--Mumford stack, as described in Remark \ref{Rem:mainA}(3) above,
	so that one could apply existing results on motivic \cite{Yasuda1,Yasuda2} as well as $ p $-adic integration \cite[Section~2]{DM_p-adic} for them.
	
	Since our focus is on Theorems~\ref{Thm:Main} and~\ref{Thm:MainIdeals}, 
	we do not go into further details here. 

\smallskip 
	
	\subsection{Summary} Let us briefly summarize the content of the paper. 
	After recalling basics on fantastacks (Section~\ref{Sec:BasicsFanta}, page~\pageref{Sec:BasicsFanta}), we treat the case of a single binomial (Section~\ref{Sec:hypersurface}, page~\pageref{Sec:hypersurface}).
	Then, we turn to the case of finitely many binomials and how the techniques extend
	to give Theorem~\ref{Thm:Main} (Section~\ref{Sec:symultaneous}, page~\pageref{Sec:symultaneous}).
	Finally, we prove Theorem~\ref{Thm:MainIdeals} in full generality (Section~\ref{binomial ideal}, page~\pageref{binomial ideal}).

\medskip 

\subsection{Acknowledgments}
For most of the explicit examples appearing in the article, 
we have used the software system {\tt polymake} \cite{Polymake} to carry out the computations on the Newton polyhedron and its dual fan.

The authors thank Anne Fr\"uhbis-Kr\"uger, Joshua Maglione, Johannes Nicaise, Ming Hao Quek, Bernard Teissier, Michael Temkin, Jeremy Usatine, and Christopher Voll
for discussions of aspects of this project. 
We thank the referee for a detailed reading and many suggestions for improvements.

\section{Basics on fantastacks}
\label{Sec:BasicsFanta}

Fantastacks are the key tool for our main theorem.
Therefore, we begin by recalling their definition following the work of Geraschenko and Satriano on toric stacks \cite{Toric1},
where fantastacks are discussed as first examples.
For more details on the theory of toric stacks, we also refer to \cite{Toric2}.
Let us point out that there exist earlier works on the notion of toric stacks by Lafforgue \cite{Laff},
	Borisov, Chen, Smith \cite{BCS}, 
	Fantechi, Mann, Nironi \cite{FMN}, and Tyomkin \cite{Tyo}, 
	each of them focusing on different aspects.
	The introduction of \cite{Toric1} provides a brief overview on them as well as a discussion how \cite{Toric1} unifies them.
There is a related treatment by Gillam and Molcho \cite{Gillam-Molcho} which allows for Deligne--Mumford stacks which are not global quotients but glued from such. We do not need this generality in this paper.

\begin{Def}[{\cite[Definition~1.1]{Toric1}}]
	Let $ X $ be a normal toric variety with torus $ T_0 $
	and let $ G $ be a subgroup of $ T_0 $. 
	The Artin stack $ [X/G] $
	equipped with the action by the dense torus $ T := T_0 / G $
	is called a {\em toric stack}.
	
\end{Def} 

For a lattice $ L $, we denote by $ L^* = \Hom_{\rm gp} ( L, \IZ) $ the dual of $ L $ as a finitely generated abelian group.  
If $ \beta \colon L \to N $ is a homomorphism of lattices, we denote by $ \beta^* \colon N^* \to L^* $ the corresponding dual.

\begin{Def}[{\cite[Definitions~2.4 and 2.5]{Toric1}}]
	\begin{enumerate}
		\item
		A {\em stacky fan} is a pair $ ( \Sigma, \beta) $ consisting of 
		a fan $ \Sigma $ on a lattice $ L $ as well as a
		homorphism of lattices  $ \beta \colon L \to N $ such that $ \beta $ has a finite cokernel.
	
		\item 
		Let $ ( \Sigma, \beta) $ be a stacky fan.
		Let $ X_\Sigma $ be the toric variety associated to $ \Sigma $ 
		and 
		let $ T_\beta \colon T_L \to T_N $ be the homomorphism of tori induced
		by $ \beta^* \colon N^* \to L^* $. 
		Set $ G_\beta := \ker(T_\beta) $.  
		The toric stack $ \cX_{\Sigma,\beta} $ associated to this data is defined as $ [X_\Sigma / G_\beta] $, where its torus is $ T_N  \cong  T_L/G_\beta $. 
	\end{enumerate}
\end{Def} 
 
As explained in \cite[after Definition~2.5]{Toric1}, 
one can construct to every toric stack $ [X/G] $ a stacky fan $ (\Sigma, \beta) $
such that $ \cX_{\Sigma,\beta} = [X/G] $.
Moreover, after these explanations, Geraschenko and Satriano give explicit examples of toric stacks.
In particular, if $ \Sigma $ is a fan on a lattice $ N $ and if we choose $ L = N $ and $ \beta = \operatorname{id}_N $, then $\cX_{\Sigma,\beta} = X_\Sigma $,
where $ X_\Sigma $ is the toric variety associated to $ \Sigma $, 
see \cite[Example~2.6]{Toric1}.

\smallskip 

We denote by $ ( e_1, \ldots, e_n ) $ the standard basis of $ \IZ^n $. 

\begin{Def}[{\cite[Definition~4.1]{Toric1}}]
	\label{Def:Fanta}
	Let $ \Sigma $ be a fan on a lattice $ N $.
	Let $ \beta \colon \IZ^n \to N $ be a homomorphism so that
	\begin{itemize}
		\item 
		the cokernel of $ \beta $ is finite,
		
		\item 
		every ray of $ \Sigma $ contains some $ \beta(e_i) $,
		and
		
		\item 
		every $ \beta(e_i) $ is contained in the support of $ \Sigma $. 
	\end{itemize}
	For a cone $ \sigma \in \Sigma $, define $ \hat\sigma := \cone( \{ e_i \mid \beta(e_i) \in \sigma \} ) $.
	Let $ \widehat{\Sigma} $ be the fan on $ \IZ^n $ generated by all these 
	$ \hat\sigma $.
	Finally, one defines $ \cF_{\Sigma, \beta} := \cX_{\widehat{\Sigma}, \beta} $.
	A toric stack which is isomorphic to some $ \cF_{\Sigma,\beta} $, is called a {\em fantastack}. 
\end{Def}

{By}~\cite[Remark~4.2]{Toric1}, one has
\begin{equation}
\label{eq:cover} 
	X_{\widehat{ \Sigma }} = \IA^n \setminus V(J_\Sigma),
	\ \ \ 
	\mbox{where } 
	J_\Sigma := \langle \prod_{\beta(e_i) \notin \sigma} x_i \mid \sigma \in \Sigma \mbox{ maximal cone } \rangle. 
\end{equation} 
In particular, $ X_{\widehat{ \Sigma }} $
(and thus also $ \cF_{\Sigma, \beta}  = [X_{\widehat{\Sigma}} / G_\beta]  $) 
is covered by the charts $ D_+ ( \prod_{\beta(e_i) \notin \sigma} x_i  ) $, where $ \sigma $ runs through all maximal cones of $ \Sigma $.
Here, $ D_+ (h) := \IA^n \setminus V(h) $ 
is the notation for the standard open set.

\smallskip 

\begin{Rk} 
	\label{Rk:morph}
	In our setting, 
	we will always have 
	that 
	$ N = \IZ^m $,
	the support of $ \Sigma $ is equal to $ \IR_\gqz^m $,
	$ n \geq m $,
	and	
	$ \beta \colon \IZ^n \to \IZ^m $
	is determined by a matrix $ M \in \IZ_\gqz^{m \times n} $ of full rank $ m $ --- in fact it will be surjective on lattices.
	In particular, $ \beta $ induces a morphism of toric stacks 
	$ \rho_\beta \colon \cF_{\Sigma,\beta} \to \IA^m $ determined by the morphism of stacky fans represented by the diagram
	\[
	\begin{array}{cr}
	\xymatrix{
		\Sigma \ar[r] & \Sigma_0
	}
	\\
	\xymatrix{
		\IZ^n \ar[d]_\beta \ar[r]^\beta & \IZ^m \ar[d]^{\rm id} \\
		\IZ^m \ar[r]^{\rm id} & \IZ^m
	}
	\end{array} 
	\]
	where $ \Sigma_0 $ is the fan on $ \IZ^m $ providing the toric variety $ \IA^m $. 
	For more details on morphisms of toric stacks and the connection to morphisms of stacky fans, we refer to \cite[Section~3]{Toric1}.
	Let us only mention that the rows of $ M $ determine $ \rho_\beta $ on the level of rings.
	For example, if	
	$ M = \begin{pmatrix}
	1 & 2 & 3 \\
	0 & 1 & 5 \\
	\end{pmatrix} $,
	then $ \rho_\beta $ arises from the restriction 
	to the open subset $ \IA^3 \setminus V(J_\Sigma) $
	(as in \eqref{eq:cover})
	of the map
	$ \IA^3 = \Spec (K[z_1, z_2, z_3] ) \to  \IA^2 = \Spec (K[x_1, x_2 ]) $,
	which is given by 
	$ x_1 = z_1 z_2^2 z_3^3 $, having exponents $(1,2,3)$ and $ x_2 = z_2 z_3^5 $ having exponents $(0,1,5)$. 
\end{Rk}

\section{Binomial hypersurfaces: generalities}

	\subsection{Setup: binomials in the torus and in affine space}\label{Sec:binomials-general} We follow the notation and terminology of Section~\ref{Sec:setup-binomials}, in particular \eqref{Eq:binomial}.

	By restriction, the description \eqref{Eq:binomial} discussed above applies to any open subset $U \subset \Spec K[x_1,\ldots,x_{m}]$.
	
	Inside the maximal torus $\Spec K[x_1^{\pm 1},\ldots,x_{{m}}^{\pm 1}]$, a binomial is of the form $$x_1^{A_1}\cdots x_r^{A_r}y_1^{-B_1} \cdots y_s^{-B_s} = \lambda,$$  a coset of the codimension-1 subgroup 
	$x_1^{A_1}\cdots x_r^{A_r}y_1^{-B_1} \cdots y_s^{-B_s} = 1.$ Its reduction is therefore always smooth, as $K$ is assumed algebraically closed
	and hence there has to exist an exponent in its equation which is not divisible by the characteristic of the ground field $ K $.
	
\subsection{Sch\"on binomials and arrangements}\label{Sec:binomials-schon}	A binomial in affine space is said to be \emph{sch\"on} if either all $A_i=0$ or all $B_j=0$. For a pure binomial this is in agreement with \cite[Definition 1.3]{Tevelev} --- note however that we impose no condition on the monomial factor, as it is an arrangement of smooth toric divisors. This means that the reduction of the purely binomial factor is smooth, and meets the toric divisor transversally.  We use this condition to define a sch\"on binomial in a smooth toric variety or stack:  

\begin{Def}\label{Def:binomial-schon}
	A hypersurface on a smooth toric stack $\cX$ is \emph{binomial} if on each affine chart it is given by an equation of the form \eqref{Eq:binomial}.
	
	Moreover, a binomial hypersurface is \emph{sch\"on} if such equations can further be chosen so that either all $A_i=0$ or all $B_j=0$.
\end{Def}

	 Note that being sch\"on still means that the reduction of the purely binomial factor is smooth, and meets the toric divisor transversally. 

	\smallskip 
	
	 Again in the torus, a collection of binomials $f_1,\ldots, f_a$ thus defines  an arrangement of codimension-1 cosets as the zero-set $X := V(f_1 \cdots f_a)$ of their product. 
	 Denote by $\cI :=  \langle f_1,\ldots, f_a \rangle $ the ideal they generate  and by $V(\cI)$ the scheme  it defines. 
	 The scheme $V(\cI)$, when nonempty, is itself a coset of a subgroup of the torus, of possibly higher codimension --- the intersection of the stabilizers of $f_i$. 
	 
	 As we assume that the base field is algebraically closed,  it follows that  the reduction $W:=V(\cI)_{\rm red}$ of such coset $V(\cI)$ is again smooth:

	 First note that $W$, as  any variety  over an algebraically closed field $K$, is generically smooth over $K$. Let $U \subset W$  be the maximal smooth open dense subset. We claim that $U = W$ and argue by contradiction.  Let $x\in W \setminus U$ and $y \in U$ be closed points. Translation by $y-x$ is an automorphism of $W$, hence it must preserve $U$, but sends the singular point $x$ to the smooth point $y$, a contradiction.

	In affine space, assume first that $f_i$ are pure binomials. Then the scheme $X = V(f_1 \cdots f_a)$ is an arrangement of   sch\"on  subschemes 
	in the sense of \cite[Definition 1.3]{Tevelev}   if and only if each binomial $f_i$ is sch\"on. 
	As above, we generalize this and say that an arrangement $X = V(f_1 \cdots f_a)$ of arbitrary binomials, not necessarily pure, on a smooth toric stack is sch\"on if each $f_i$ is sch\"on. 
	
	If the binomials $f_i$ are sch\"on and have trivial monomial factor, the binomial ideal $\cI= \langle f_1,\ldots,f_a \rangle $ they generate defines a sch\"on subscheme  in the sense of \cite[Definition 1.3]{Tevelev}. The situation is a bit more involved when the monomial factors are not trivial, see Section \ref{binomial ideal}.

	\subsection{Simple arrangements and problematic primes} \label{Sec:simple-arrangement}
	
	\begin{Def} 
	Given a smooth Artin stack $ \cA $ and a substack $ Y \subseteq \cA $, 
	we say that $ Y $ is a {\em simple arrangement on $\cA $}, 
		if the following conditions hold for
		the intersection lattice $ \cS = \{ S_i \}_{i \in I} $ of $  Y $:
		\begin{enumerate}
			\item 
			$ S_i $ is smooth and non-empty,
			
			\item
			$ S_i $ and $ S_j $ \emph{meet cleanly.} 
			This is defined to mean that
			 their scheme-theoretic intersection is smooth and their tangent spaces fulfill $ T(S_i) \cap T(S_j) = T(S_i \cap S_j) $, 
			and
			
			\item 
			$ S_i \cap S_j = \varnothing $ or $ S_i \cap S_j $ is a non-empty, disjoint union of elements $ S_\ell \in \cS $,
	\end{enumerate}
	for every $ i , j \in I $.
	\end{Def} 

	In characteristic 0, a sch\"on arrangement of binomials on a toric stack is automatically a simple arrangement. This is not the case in positive characteristic. The precise condition is easily computed by taking derivatives:

	\begin{Obs}
		\label{Obs:problematic}
	As in \cite[Section 6]{Teissier} it suffices to check that each $S_i$ is (absolutely) reduced. 
Thus, consider a set of distinct binomials $f_1,\ldots,f_a$, of the form $f_j = \prod_{i=1}^{{m}} x_i^{D_{ij}} - \lambda_j$ with $D_{ij} \in \IZ$, vanishing in codimension $k$.  
For the common vanishing locus to be smooth, the matrix $\bar D_{ij} \in M_{{m,a}}(\IF_p)$ obtained by reducing the integers $D_{ij}$ modulo $p$ must be of rank $k$. To see this, first note that the sch\"on condition forces the variables $x_i$ that appear to be nonzero, hence we may work in the torus. We need to compute the rank of the derivative matrix with respect to all $x_i$, and since the variables are invertible we may instead use the logarithmic derivatives $x_i \partial/\partial x_i$. Note that $x_i \partial f_j/\partial x_i  = D_{ij} \prod_{i=1}^{{m}} x_i^{D_{ij}}.$ Factoring out the invertible monomials, it remains to compute the rank of the matrix $\bar D_{ij}$ as claimed. Putting these conditions together, we require that for each  subset $\{f_j, j \in I\}$ of the given binomials, vanishing in codimension $k_I$, the corresponding matrix $ (\bar D_{ij})_{i \in \{ 1,\ldots, m \}, j\in I}  
\in M_{{m, |I|}}(\IF_p)$ must be of rank $k_I$.
	\end{Obs}

	  Note that given an integer matrix $D$ describing a collection of pure binomials in the torus, there are only finitely many primes where the requisite minors $\det  (\bar D_{ij})_{i \in \{ 1,\ldots, m \}, j\in I}$ above vanish. 
	  We denote this collection of \emph{problematic} primes $\cE$. Thus  the sch\"on arrangement defined by $f_i$ fails to be simple if and only if $p\in \cE$ is problematic.

\subsection{Subdivisions and modifications associated to binomials}
Consider again a binomial $f = x_1^{C_1}\cdots x_{{m}}^{C_{{m}}}\left(x_1^{A_1}\cdots x_r^{A_r} -  \lambda y_1^{B_1} \cdots y_s^{B_s}\right)$ in affine space. The fan $\Sigma_0$ of affine space has a single maximal cone $\sigma = \IR_{\geq 0}^{{m}}$.  We denote by $\xi_i$ the coordinates in $\IR^m$ corresponding to $x_i$ and $\eta_j$ those corresponding to $y_j$. The binomial thus defines a hyperplane $h_f(\xi_i, \eta_j) =  \sum A_i \xi_i - \sum B_j \eta_j = 0$ which splits $\sigma$ into a fan $\Sigma_f$ of two cones; 
note that the monomial factor $ x_1^{C_1}\cdots x_{{m}}^{C_{{m}}}$ does not affect this subdivision. 
In modern terms, the hyperplane $h_f(\xi_i, \eta_j)$ is the tropicalization of the binomial $f$.\footnote{We thank the referee for pointing this out. Note that the works \cite{Teissier} and \cite{Eisenbud-Sturmfels} predate this notion!} 
As in Section~\ref{Sec:GPT}, any smooth subdivision of $\Sigma_f$ results in a resolution of $\{f=0\}$, and here we avoid these subdivisions using stacks.

A collection $f_1,\ldots,f_a$ of binomials provides a collection of hyperplanes $h_{f_j} = 0$. 
Inductively, each hyperplane subdivides any cone into at most two subcones, hence together they divide $\sigma$ into a fan $\Sigma_f$ of at most $2^a$ cones. The subdivision has $<\, \binom{{m}+a}{{m}-1}$ edges, as an edge is determined by the intersection of $({m}-1)$ hyperplanes chosen from the union of the $ {m} $ coordinate hyperplanes and $h_{f_j} $. Note that these bounds depend only on {the ambient} dimension and {the} number of hypersurfaces, and not on the exponents appearing.

For a subdivision $\Sigma \to \Sigma_0$ we write $M$ for the matrix whose columns are the primitive generators of the rays of $\Sigma$. It defines a homomorphism $\beta:  \IZ^n \to \IZ^m$, resulting in a fantastack $\cF_{\Sigma,\beta}$ and a morphism of toric stacks  $ \rho = \rho_\beta \colon  \cF_{\Sigma,\beta} \to \IA^m $. 
We show that this  stack theoretic modification desingularizes the $f_i$ simultaneously.

\section{Resolving a single binomial hypersurface}

\label{Sec:hypersurface}

In this section, we prove Theorems~\ref{Thm:Main} and~\ref{Thm:MainIdeals} for a single binomial hypersurface,
where they coincide.
Let $ K $ be an algebraically closed field and let 
	$ f $ be a binomial with coefficients in $ K $.
In the hypersurface case, 
the stacky fan $ ( \Sigma, \beta  ) $ for the fantastack construction will arise from the dual fan of the Newton polyhedron of $ f $. 
In particular, we have $ N = \IZ^m $, for some $ m \in \IZ_{\geq 2} $.
Furthermore, a monomial factor dividing $ f $ will not have an effect.
Therefore, we assume without loss of generality that
$ f $ is a pure binomial,
\[  
	f  = x^A - \lambda y^B = x_1^{A_1} \cdots x_s^{A_s}  - \lambda y_1^{B_1} \cdots y_r^{B_r} 
	\in K [x,y]
	,
\]
where $ \lambda \in K \setminus \{ 0 \} $.
Set $ m := s + r $.

Recall \cite{Kouchnirenko}, \cite[Section~18.2]{CJS} that the {\em Newton polyhedron}
$ \NP(g;z) $ of a polynomial $ g \in K[z] = K[z_1, \ldots, z_m] $ is defined as follows:
	If we have 
	\[ 
	g  = \sum_{C\in \IZ^{m}_{\geq 0}} \lambda_C z_1^{C_1} \cdots z_m^{C_m} 
	\]
	with coefficients $ \lambda_C \in K $, 
	then 
	$ \NP(g;z) $ is the smallest closed convex subset containing all points of the set 
	\[
		\{ C + v \mid v \in \IR_\gqz^m \text{ and }  C \in \IZ_\gqz^m \text{ is such that } \lambda_C \neq 0  \} .
	\] 
In particular, 
$ \NP(f;x,y) $ is 
the smallest closed convex subset $ \Delta \subseteq \IR_\gqz^m $ 
containing the points $ (A,0), (0,B) \in \IZ_\gqz^m $ and 
which is stable under translations by the non-negative orthant $ \IR_\gqz^m $, 
i.e., 
$ \Delta + \IR_\gqz^m = \Delta $. 
When there is no confusion possible, 
we sometimes write $ \NP(f) $ instead of $ \NP(f;x,y) $. 
We will use the symbol $ \Sigma_f  $ for the {\em dual fan} of the Newton polyhedron $ \NP(f) $, see \cite{Varchenko,MingHao},
which is the fan in $ \IR^m $, whose
cones are determined by the normal vectors of the facets of $ \NP(f) $.
(For some explicit examples and pictures, we refer to the examples later on.)
Observe that $ e_1, \ldots, e_m $ appear among the normal vectors of 
$ \NP(f) $ and  
the support of the dual fan $  \Sigma_f $ is equal to $ \IR_\gqz^m $. 

\smallskip 

Let us introduce the homorphism 
$ \beta \colon \IZ^n \to \IZ^m $,
which we use for the construction of the fantastack $ \cF_{\Sigma,\beta}$,
where $ \Sigma $ is a subdivision of $ \Sigma_f $.

\begin{Def}
	\label{Def:Matrix}
	Let $ f = x^A - \lambda y^B \in K[x,y] $ be a  binomial
	as above, 
	and let $ \Sigma $ be any subdivision of the dual fan $ \Sigma_f $ of the Newton polyhedron of $ f $;
	note that $ \Sigma = \Sigma_f $ is allowed. 
	Let $ e_1, \ldots, e_m , v_{m+1}, \ldots, v_n \in \IZ_\gqz^m $ be primitive generators for the rays of $ \Sigma $.
	Notice that the normal vectors of the facets of the Newton polyhedron $ \NP(f;x,y) $ are among them.  
	Let $ M_\Sigma $ be the $ m \times n $ matrix whose columns are the primitive generators 
	\[
		M_\Sigma = \begin{pmatrix}
		e_1 & \cdots & e_m & v_{m+1} & \cdots & v_n 
		\end{pmatrix}
		= \begin{pmatrix}
		E_m &  v_{m+1} & \cdots & v_n 
		\end{pmatrix} \in \IZ_\gqz^{m\times n },
	\]
	where $ E_m \in \IZ^{m\times m} $ is the unit matrix.  
	We define $ \beta_\Sigma \colon \IZ^{n} \to \IZ^{m} $ to be the homorphism determined by the matrix $ M_\Sigma $.
	If the reference to $ \Sigma $ is clear from the context, we sometimes just write $ \beta $ instead of $ \beta_\Sigma $. 
\end{Def}

Since $ M_\Sigma $ has full rank, $ \beta_\Sigma $ is surjective. 
In particular, its cokernel is finite. 
Furthermore, 
by construction,
every ray of $ \Sigma $ contains some $ \beta_\Sigma (e_i) $ and every $ \beta_\Sigma(e_i) $ is contained in the support of $ \Sigma $. 
Therefore, $ (\Sigma, \beta) $ fulfills all hypotheses of Definition~\ref{Def:Fanta}.
Recall that $ \beta $ induces a morphism of toric stacks $ \rho_\beta  \colon \cF_{\Sigma,\beta} \to \IA^m $
	(Remark~\ref{Rk:morph}).

\medskip 

In the given situation, Theorems~\ref{Thm:Main} and~\ref{Thm:MainIdeals} boil down to the following statement.

\begin{Prop}
	\label{Prop:red}
	Let $ K $ be an algebraically closed field and let $ f $ be a binomial with coefficients in $ K $.
	Let $ \Sigma $ be any subdivision of the dual fan of the Newton polyhedron of $ f $.
	Let $ \cF_{\Sigma,\beta} $ be the fantastack associated to the stacky fan $ (\Sigma,\beta) = (\Sigma, \beta_\Sigma) $
	and 
	let 
	$ \rho_\beta \colon \cF_{\Sigma,\beta} \to \IA^m $ be the morphism induced by $ \beta $.
	
	If we set $  X := V (f) \subset \IA^m $,
	then $
	\rho_\beta^{-1}(X)_{\rm red} $ is a sch\"on binomial hypersurface (Definition \ref{Def:binomial-schon}) and a simple normal crossing divisor on $ \cF_{\Sigma,\beta} $.
\end{Prop}

Let us first discuss this in an example to explain our intuition. 
	The proof is given in Section~\ref{Pf:3.2} on page \pageref{Pf:3.2}.

\begin{Ex}
	Consider the binomial $ f = x_1^2 x_2^2  - y^ 3 $. 
	The Newton polyhedron looks as follows:
	
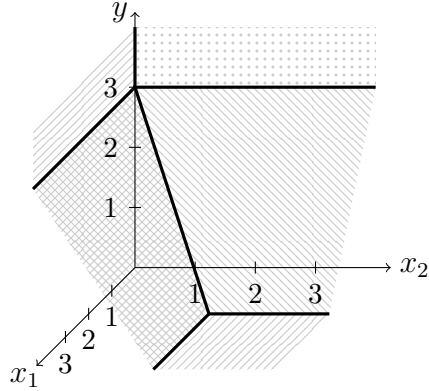
\[
	\begin{tikzpicture}[scale=0.8]
	
	\path[pattern=north west lines, pattern color=black!20!white, dashed] 
	(4,0,2)--(2,0,2) --(0,3,0) -- (4,3,0) -- (4,0,2);
	
	\path[pattern=north east lines, pattern color=black!20!white, dashed] 
	(2,0,4.4)--(2,0,2) --(0,3,0) -- (0,3,4.4) -- (2,0,4.4);
	\path[pattern=north west lines, pattern color=black!20!white, dashed] 
	(2,0,4.4)--(2,0,2) --(0,3,0) -- (0,3,4.4) -- (2,0,4.4);
	
	\path[pattern=north east lines, pattern color=black!20!white, dashed] 
	(0,3,0) -- (0,4,0) -- (0,4,4.4) -- (0,3,4.4) -- (0,3,0);

	\path[pattern=north east lines, pattern color=black!20!white, dashed] 
	(2,0,4.4)--(2,0,2) --(4,0,2) -- (4,0,4.4) -- (2,0,4.4);
	
	\path[pattern=dots, pattern color=black!20!white, dashed] 
	(0,3,0) -- (4,3,0) -- (4,4,0) --  (0,4,0);

	\draw[<->] (0,4.25,0)--(0,0,0)--(4.25,0,0);
	\draw[->] (0,0,0)--(0,0,4.25);	
	
	\node at (0,0,4.75) {$ x_1 $};
	\node at (4.65,0,0) {$ x_2 $};
	\node at (-0.25,4.25,0) {$ y $};
	
	\foreach \x in {1,...,3}
	\draw (\x,0.1,0) -- (\x,-0.1,0) node [below] {\small \x};
	
	\foreach \x in {1,...,3}
	\draw (0.1,\x,0) -- (-0.1,\x,0) node [left] {\small  \x};
	
	\foreach \x in {1,...,3}
	\draw (0,0.1,\x) -- (0,-0.1,\x) node [below] {\small  \x};
	
	\draw[very thick] (2,0,2) -- (0,3,0)--(4,3,0);
	\draw[very thick] (4,0,2) -- (2,0,2) -- (2,0,4.4);
	\draw[very thick] (0,3,4.4) -- (0,3,0) -- (0,4,0);

	\end{tikzpicture} 
	\] 
	Here we highlighted the 1-skeleton and the facets of the polyhedron, 
	and the polyhedron itself forms the solid bounded by the resulting 5 facets and closer to the observer. 
	
	The normal vectors of its facets are $ e_1, e_2, e_3, v, w $,
	where 
	$ u = (3,0,2)^T $ and $ v = (0,3,2)^T $.   
	Therefore, the dual fan of the Newton polyhedron looks as follows
	
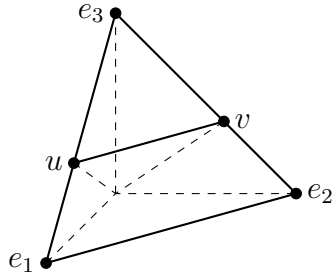
\[
	\begin{tikzpicture}[scale=0.8]
	
	\draw[dashed] (0,0,3)--(0,0,0)--(3,0,0);
	\draw[dashed] (0,0,0)--(0,3,0);
	\draw[dashed] (1.8,1.2,0)--(0,0,0)--(0,1.2,1.8);

	\filldraw[black] (3,0,0) circle (2.5pt) node[right] {$ e_2 $};
	\filldraw[black] (0,3,0) circle (2.5pt) node[left] {$ e_3 $};
	\filldraw[black] (0,0,3) circle (2.5pt) node[left] {$ e_1 $};
	\filldraw[black] (1.8,1.2,0) circle (2.5pt) node[right] {$ v $};
	\filldraw[black] (0,1.2,1.8) circle (2.5pt) node[left] {$ u $};
	
	\draw[thick] (0,0,3) -- (0,3,0) -- (3,0,0) -- (0,0,3);
	\draw[thick] (0,1.2,1.8) -- (1.8,1.2,0);
	
	\end{tikzpicture} 
	\]
	Here we highlighted a transverse slice of the fan. The plane spanned by the vectors  $u$ and $v$ is precisely the plane $2\xi_1 + 2\xi_2 = 3 \eta$ described in the introduction and appearing in the paper \cite{GPT}. 
	
	Take $ \Sigma = \Sigma_f $ to be the dual fan of the Newton polyhedron itself, without further modification.
	The columns of the matrix $ M := M_\Sigma $ are the normal vectors above,
	\[
		M = 
		\begin{pmatrix}
			1 & 0 & 0 & 3 & 0 \\
			0 & 1 & 0 & 0 & 3 \\
			0 & 0 & 1 & 2 & 2 \\
		\end{pmatrix}.
	\] 
	Note that the kernel of $ M $ has the basis 
	$  
		( (-3,0,-2,1,0)^T , (0,-3,-2,0,1)^T ) .
	$
	The rows of $ M $ determine the morphism $ \rho  \colon \IA^5 \to \IA^3 $ given by 
	\[
		\begin{array}{ccc}
			x_1 = x_1'  z_1^3 ,
			&	
			x_2 = x_2' z_2^3,
			& 
			y = y' z_1^2 z_2^2 , 
		\end{array}
	\]
	where $ x_1', x_2', y', z_1, z_2 $ are the variables of $ \IA^5 $. 
	We observe that the binomial $ f $ becomes 
	\[
		f = z_1^6 z_2^6 ( x_1'^2 x_2'^2 - y'^3 ) .
	\]
	In the construction of $ \cF_{\Sigma,\beta} $, we have to take the quotient 
	by $ G_\beta = \ker(T_\beta) $, where $ T_\beta \colon T_{\IZ^n}  \to T_{\IZ^m} $ is the homomorphism induced by $ \beta $.
	In this example,  $G_\beta\simeq \IG_m^2 $ acts via
	\[
		(t_1,t_2) \cdot
		(x_1',x_2',y', z_1,z_2) 
		:=
		(
		t_1^{-3} x_1',
		\
		t_2^{-3} x_2',
		\
		t_1^{-2} t_2^{-2} y',
		\
		t_1 z_1,
		\
		t_2 z_2 
		).
	\]
	Indeed, writing 
		\[ (a_{x_1}, a_{x_2}, a_{y}, a_{z_1}, a_{z_2}) := (-3,0,-2,1,0) 
		\ \ \mbox{ and } \ \   
	 (b_{x_1}, b_{x_2}, b_{y}, b_{z_1}, b_{z_2}) := (0,-3,-2,1,0) \] 
	for the basis of $ \ker(M) $, we have that $ t_1^{-3} x_1' = t_1^{a_{x_1}} t_2^{b_{x_1}} x_1' $ etc.
	Observe that the original variables $ (x_1, x_2, y) $ are stable under the action,
	by construction.

	By \eqref{eq:cover}, $ \cF_{\Sigma, \beta} $ is covered by the two charts
	$ D_+ (x_1'x_2') $ and $ D_+ (y') $. 
	In both charts 
	--- i.e., if $ x_1' x_2' \neq 0 $ or if $ y' \neq 0 $ ---
	the polynomial $ x_1'^2 x_2'^2 - y'^3 = 0 $ defines a smooth hypersurface.

	Let us consider a further subdivision, in which we have the additional ray $ w = (2,1,0)^T $,
		e.g.,
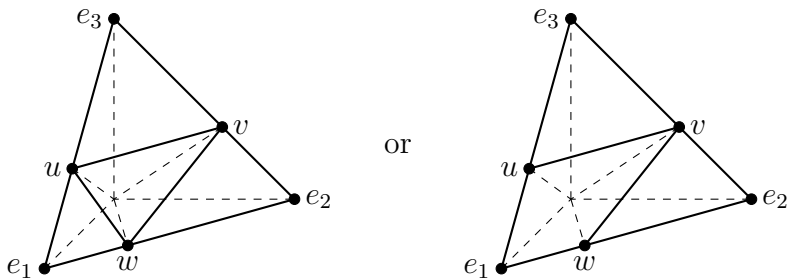
		\[
		\begin{tikzpicture}[scale=0.8]
		
		\draw[dashed] (0,0,3)--(0,0,0)--(3,0,0);
		\draw[dashed] (0,0,0)--(0,3,0);
		\draw[dashed] (0,0,0)--(1,0,2);
		\draw[dashed] (1.8,1.2,0)--(0,0,0)--(0,1.2,1.8);

		\filldraw[black] (3,0,0) circle (2.5pt) node[right] {$ e_2 $};
		\filldraw[black] (0,3,0) circle (2.5pt) node[left] {$ e_3 $};
		\filldraw[black] (0,0,3) circle (2.5pt) node[left] {$ e_1 $};
		\filldraw[black] (1.8,1.2,0) circle (2.5pt) node[right] {$ v $};
		\filldraw[black] (0,1.2,1.8) circle (2.5pt) node[left] {$ u $};
		
		\filldraw[black] (1,0,2) circle (2.5pt) node[below] {$ w $};
		
		\draw[thick] (0,0,3) -- (0,3,0) -- (3,0,0) -- (0,0,3);
		\draw[thick] (0,1.2,1.8) -- (1.8,1.2,0) -- (1,0,2) -- (0,1.2,1.8);
		
		\end{tikzpicture} 
		\ \ \ 
		\raisebox{1.8cm}{ or } 
		\ \ \
		\begin{tikzpicture}[scale=0.8]
		
		\draw[dashed] (0,0,3)--(0,0,0)--(3,0,0);
		\draw[dashed] (0,0,0)--(0,3,0);
		\draw[dashed] (0,0,0)--(1,0,2);
		\draw[dashed] (1.8,1.2,0)--(0,0,0)--(0,1.2,1.8);

		\filldraw[black] (3,0,0) circle (2.5pt) node[right] {$ e_2 $};
		\filldraw[black] (0,3,0) circle (2.5pt) node[left] {$ e_3 $};
		\filldraw[black] (0,0,3) circle (2.5pt) node[left] {$ e_1 $};
		\filldraw[black] (1.8,1.2,0) circle (2.5pt) node[right] {$ v $};
		\filldraw[black] (0,1.2,1.8) circle (2.5pt) node[left] {$ u $};
		
		\filldraw[black] (1,0,2) circle (2.5pt) node[below] {$ w $};
		
		\draw[thick] (0,0,3) -- (0,3,0) -- (3,0,0) -- (0,0,3);
		\draw[thick] (0,1.2,1.8) -- (1.8,1.2,0) -- (1,0,2);
		
		\end{tikzpicture} 
		\]
		In the matrix $ M $, we have to add the extra column $ w $.
		Analogous to above, one may determine the corresponding morphism of toric stacks and verify that the preimage of the binomial defines a simple normal crossings divisor. 
		Notice that there are more maximal cones compared to above and hence, we have to consider more charts.  On the other hand the fan on the left, which is a triangulation of the fan on the right, is not much more complex: the chart on the right corresponding to $\langle e_1,u,v,w\rangle$ is defined by $x'_2y'$ being invertible. On the left we have its two open substacks where either $x_1'x'_2y'$ or $z_2'x'_2y'$ is invertible.
\end{Ex}

\begin{Rk}
	Recall that by \eqref{eq:cover},
	we are working in an open subset of $ \IA^5 $ in the example. 
	If one is only interested in the transform of
	$ X =  V (x_1^2 x_2^2 - y^3 ) $,
	then it is sufficient 
	to consider one of the charts $ D_+ (x_1'x_2') $ and $ D_+ (y') $
	since the transform is contained in the intersection of both.
	On the other hand, if one needs to keep track how the ambient space behaves, e.g.~for integration, one needs all of its charts.
\end{Rk}

\subsection{Proof of Proposition~\ref{Prop:red}}
\label{Pf:3.2}
	Let $ f $ be a binomial.
		If the preimage of $ f $ provides a simple normal crossing divisor, necessarily of the form $x^C(x^A - \lambda)$, when applying the fantastack construction for $ \Sigma_f $,
		then any further subdivison $ \Sigma $ of $ \Sigma_f$ cannot change this. 
		Therefore, it suffices to consider the case $ \Sigma = \Sigma_f $.
	 
	 As before, we assume without loss of generality that
	$  f = x^A - \lambda y^B $, 
	where $ A = (A_1, \ldots, A_s) $ and $ B = (B_1, \ldots, B_r ) $. 
	If $ A  = 0 $ or $ B = 0 $, then $ V(f) $ is smooth and there is nothing to prove. 
	Hence, we assume $ { A \in \IZ_+^s } $ and $ { B \in \IZ_+^r } $ in the following.
	We prove the result by giving an explicit description of the matrix $ M = M_\Sigma  $:
	Without loss of generality, 
	we assume that all entries of $ A $ and $ B $ are non-zero. 
	For every $ i \in \{ 1, \ldots, s \} $, we define 
	\[ 
		D(A_i) \in \IZ^{r \times r} 
	\] 
	to be the diagonal matrix with all entries to be $ A_i $.
	For every $ \ell \in \{ 1, \ldots, s \} $,
	we introduce  
	\[
		R_\ell (B) \in \IZ^{s \times r} 
	\]
	to be the matrix, whose $ \ell $-th row is equal to $ B $ and all other entries are zero.
	For example, if $ s = 3 $, $ r = 4 $, $ \ell = 2$, we have 
	\[
		R_2 (B) = \begin{pmatrix}
		0 & 0 & 0 & 0 \\
		B_1 & B_2 & B_3 & B_4 \\
		0 & 0 & 0 & 0 \\
		\end{pmatrix}
		.
	\]
	We denote by $ E_\alpha \in \IZ^{\alpha \times \alpha} $ the unit matrix and
	by $ 0^{\alpha \times \beta} \in \IZ^{\alpha \times \beta } $ we mean the $ \alpha \times \beta $ matrix with all entries zero, where $ \alpha, \beta \in \IZ_+ $.
	Using this notation, we claim that the facets of the Newton polyhedron $ \NP(f) = \NP(f;x,y) $ provide the matrix
		\[
		M' = 
		\begin{pmatrix}
		E_s & 0^{s\times r}
		& R_1 (B) & R_2(B) 
		& \cdots & R_s (B) 
		\\
		0^{r \times s} 
		& E_r &
		D(A_1) & D(A_2) & \cdots & D(A_s)
		\end{pmatrix} 
		\in \IZ^{m \times n  },
		\]
		where we recall that $ m =  r + s $ and we set $ n := r+s+rs $.
		In general the columns of $M'$ might not be primitive. Dividing  each column by the greatest common divisors of its entries we obtain the desired matrix $M$.\footnote{We note however that $M'$ will also give a fantastack resolution, obtained by a suitable root stack of $\rho$, see \cite{MingHao}.}
		
		Let us explain this:
		For $ m = 2 $, the statement can be easily verified since $ \NP(f) $ has only three facets.
		Thus, assume $ m \geq 3 $.
		Since we have $ \NP(f) + \IR_\gqz^{m} = \NP(f) $, 
		the Newton polyhedron $ \NP(f) $ has unbounded facets parallel to the coordinate hyperplanes of $ \IR^{m} $.
			In other words, 
		all unit vectors $ e_1, \ldots, e_{m} $ appear as normal vectors.
		Further, notice that the vertices of $ \NP(f) $ are $ (A,0) $ and $ (0,B) $. 
		The columns of $ M $, 
		not coming from $ E_{s+r} $,
		are of the form 
		\[ 
			v_{i,j} := B_{j-s} e_i + A_i e_j ,
		\ \ \ 
		\mbox{ for } i \in \{ 1, \ldots, s \}, 
		j \in \{ s+1, \ldots, s+r \} .
		\]
		Fix such $ i, j $ and let $ k \in \{ 1, \ldots, m \} \setminus \{ i,j \} $ be any other element in $ \{ 1 , \ldots, m \} $ different from $ i $ and $ j $.
		Then, the following equalities hold
		\begin{equation}
			\label{eq:equalities}
		(A,0) \cdot v_{i,j}  = (0,B) \cdot v_{i,j}  = ((A,0) + e_k ) \cdot v_{i,j}  = ((0,B)+e_k) \cdot v_{i,j}
			= A_i B_{j-s} .
		\end{equation}
		Note that the face of $ \NP(f) $ defined by $ v_{i,j} $ is
			\[
			\cF_{i,j} := \{ w \in \NP(f) \mid w \cdot v_{i,j} = A_i B_{j-s} \}
			.
			\]
			Let $ \mathfrak S $ be the segment 
			determined by $ (A,0) $ and $ (0,B) $.
			The latter can be described as
			\[
				\mathfrak S = \{  \varepsilon (A,0) + (1-\varepsilon) (0,B)  \mid \varepsilon \in [0;1] \subset \IR \}. 
			\]
			By \eqref{eq:equalities}, we have 
			\[
			 	w \cdot v_{i,j} = A_i B_{j-s},
			 \ \ \ \mbox{ for every } w \in \mathfrak S. 
			\]
			In other words, $ \mathfrak S \subseteq \mathcal{F}_{i,j} $.
			Since further \eqref{eq:equalities} holds for all $ k \in \{ 1, \ldots, m \} \setminus \{ i,j \} $, we obtain that
			$ \cF_{i,j} $ has dimension $ m - 1 $ and thus is a facet.
		
		It remains to prove that we determined all facets of $ \NP(f) $. 
		Suppose there exists some $ v = (v_1, \ldots, v_m ) \in \IR_\gqz^m $,
		which is a normal vector of $ \NP(f) $
		different from those discussed before.
		Let $ \alpha := \# \{ k \in \{ 1, \ldots, m \} \mid v_k \neq 0 \} \geq 2 $ 
		and let $ v_{k_1}, \ldots, v_{k_\alpha} $ be the non-zero entries of $ v $,
		where $ k_1 < k_2 < \ldots < k_\alpha $.
		If $ \alpha = m $, 
		then $ v $ determines a compact face of $ \NP(f) $.
		The only compact faces are the two vertices and the segment connecting them. 
		Since $ m \geq 3 $, none of them is a facet.
		Hence, suppose that $ \alpha < m $.		
		For $ \alpha \geq 3 $, 
		consider the projection $ \pi_v \colon \IR^m \to \IR^\alpha,  (w_1, \ldots, w_m) \mapsto  (w_{k_1}, \ldots, w_{k_\alpha}) $.
		Then, $ \pi_v (v) $ determines a compact face of $ \pi_v ( \NP(f)) $
		and the latter has to be the segment connecting the two vertices.
		Since $ \alpha \geq 3 $ and since all entries of $ A $ and $ B $ are non-zero, 
		the preimage of the compact face along $ \pi_v $
		does not correspond to a facet of $ \NP(f) $.
		Finally, if $ \alpha = 2 $,
		the equality $ (A,0)\cdot v = (0,B) \cdot v  $ implies
		$ \frac{v_{k_1}}{v_{k_2}} =  \frac{B_{k_2-s}}{A_{k_1}} $
		and thus $ v $ is a multiple of $ v_{k_1, k_2} $.

		{Next, let us describe the columns of $ M $. 
		For $ i \in \{ 1, \ldots, m \} $, the $ i $-th column is the $ i $-th unit vector of $ \IZ^{m} $.
		Fix $ k \in \{ 1, \ldots, s \} $, $ \ell  \in \{ 1, \ldots , r \} $ and set $ i := i_{k,\ell} := m + (k-1) r + \ell $.
		The non-zero entries of the $ i $-th column of $ M' $ are $ B_\ell $ and $ A_k $.
		Let $ d_{k,\ell} $ be the greatest common divisor of $ B_\ell $ and $ A_k $. 
		Define $ B_{k,\ell} := B_\ell/ d_{k,\ell} $ and $ A_{k,\ell} := A_k/ d_{k,\ell}  \in \IZ_\gqz $.
		Then, the non-zero entries of the $ i $-th column of $ M $ are $ B_{k,\ell} $ and $ A_{k,\ell} $.}

	The matrix $ M $ induces $ \beta =  \beta_f \colon \IZ^n \to \IZ^m $
	and thus the homomorphism $ \rho_\beta \colon \IA^{n} \to \IA^{{m}} $,
	which is given by 
	\[
		\begin{array}{ll}
		\displaystyle  
		x_i = x_i' z_{i,1}^{{B_{i,1}}} z_{i,2}^{{B_{i,2}}} \cdots z_{i,r}^{{B_{i,r}}}
		{= x_i' \prod_{j=1}^r z_{i,j}^{B_{i,j}}},
		&
		\mbox{for } i \in \{ 1, \ldots, s \}, 
		\\[15pt]
		\displaystyle 
		y_j = y_j' z_{1,j}^{{A_{1,j}}} z_{2,j}^{{A_{2,j}}} \cdots z_{s,j}^{{A_{s,j}}}
		{= y_j' \prod_{i=1}^s  z_{i,j}^{A_{i,j}}},
		&
		\mbox{for } j \in \{ 1, \ldots, r \},
		\end{array}
	\]
	where 
	$
	x_1', \ldots, x_s', y_1', \ldots, y_r', z_{1,1}, \ldots , z_{1,r}, z_{2,1}, \ldots, z_{2,r}, \ldots  z_{s,r} 
	$
	are the variables of $ \IA^n $.
	{Note that the variable $ z_{k,\ell} $ corresponds to the $ i_{k,\ell} $-th column of $ M $
		with $ i_{k,\ell} = m + (k-1) r + \ell $.} 
	We obtain
	\[  
		x^A -  \lambda y^B = 
		{\prod_{i=1}^s x_i^{A_i} - \lambda \prod_{j=1}^r y_j^{B_j} = 
		\prod_{i=1}^s \Big( x_i'^{A_i} \prod_{j=1}^r z_{i,j}^{A_i B_{i,j}} \Big)- \lambda \prod_{j=1}^r \Big( y_j'^{B_j} \prod_{i=1}^s z_{i,j}^{A_{i,j} B_j} \Big)  = }
	\]
	\[
		= \Big( \prod_{i=1}^s \prod_{j=1}^r z_{i,j}^{{d_{i,j} A_{i,j}  B_{i,j}}} \Big) \cdot  ( x'^{A} - \lambda y'^{B} ) .
	\]
	The fantastack $ \cF_{\Sigma_f, \beta} $ is covered by the two charts $ D_+ (\prod_{i=1}^s x_i') $ and $ D_+ (\prod_{j=1}^r y_j') $. 
	Since $  x'^{A} - \lambda y'^{B} $ is a sch\"on smooth binomial in both charts, 
	this ends the proof.
	\hfill $ \square $

\begin{Rk}
	\label{Rk:non_red}
	Let $ f $ be a product of binomials, 
		$ f = \prod_{i=1}^d ( x^{A} - \lambda_i y^{B}) $,
		for some $ \lambda_i \in K \smallsetminus \{ 0 \} $,
	where the  exponents $ A,B $ appearing are \emph{the same} for each factor.
	Then, the same procedure as for $ V(x^{A} - \lambda y^{B}) $ transforms $ V (f) $ into a simple normal crossings divisor.
	Notice that,
	for $ \lambda \neq \lambda' $, 
 	$ V (x^{A} - \lambda y^{B}) $ and $ V(x^{A} - \lambda' y^{B} )$ 
	are disjoint if $ x_i \neq 0 $ for all $ i $
	(or $ y_j \neq 0 $ for all $ j $).
\end{Rk}

\section{Partial simultaneous resolution of finitely many binomials}
\label{Sec:symultaneous}

In Remark~\ref{Rk:non_red}, 
we have seen that we can simultaneously resolve the singularities of finitely many binomial varieties $ V(x^A - \lambda_i y^B) $ via a single step using fantastacks, 
where $ \lambda_i \in K $ are  pairwise different and the pairs of exponents $ (A,B) $ are the same for all $ i $.
Our next goal is to extend this result without restrictions on the exponents. 
In particular, we prove Theorem~\ref{Thm:Main} in full generality, 
for finitely many binomials.

\begin{Obs}
	\label{Obs:Simul}
	Let $ f_1, \ldots, f_a \in K[x] = K[x_1, \ldots, x_m ] $ be finitely many binomials.
	\begin{enumerate}
		\item 
		For every $ i \in \{ 1 , \ldots, a \} $, we determined a matrix $ M(f_i) \in \IZ^{m \times n_i} $ via the normal vectors of the facets of $ \NP(f_i) $ in the previous section, 
		which induced a transformation $ \rho_i \colon \IA^{n_i} \to \IA^m $ resolving the singularities of $ V(f_i) $. 
		In general, it is not sufficient
		to consider the transformation induced by the matrix whose set of columns coincides with the union of the normal vectors of $ \NP(f_i;x) $, for $ i \in \{ 1 , \ldots , a\} $,  
		see Example~\ref{Ex:simul}(2).
		
		\item 
		The task to simultaneously monomialize $ V(f_1), \ldots, V(f_a) $  is equivalent to the monomialization of the hypersurface determined by their product $ V(f_1 \cdots f_a )$.
		
		\item 
		From the definition of the {Minkowski} sum, it follows directly that
		the Newton polyhedron of the product $ f_1 \cdots f_a $ is the Minkowski sum of the Newton poyhedra for $ f_i $,
		\[
			\NP(f_1 \cdots f_a;x) =
			\NP(f_1;x) + \cdots + \NP(f_a;x),
		\]
		where ``$ + $" is the Minkowski sum.		
		In particular, 
		any vertex $ v $ of $ \NP(f_1 \cdots f_a)  $
		can be decomposed as $ v  = v_1 + \ldots + v_a $, 
		where $ v_i \in \NP(f_i) $,
		for $ i \in \{ 1 , \ldots, a \} $.
		Note that the converse of the last statement is not true, 
		i.e., not every sum of vertices $ v_1 + \ldots + v_a $ is also a vertex of $ \NP(f_1 \ldots f_a ) $,
		see Example~\ref{Ex:simul}(1).
		
	\end{enumerate} 
\end{Obs}

Let us discuss some examples with two binomials, where we write $ f = f_1 $ and $ g = f_2 $.

\begin{Ex}
	\phantomsection
	\label{Ex:simul}
	\begin{enumerate}
		\item 
	Consider the binomials $ f = x^2 - y^3 $ and $ g = x^2 - y^5 $.
	We have $ f g = x^4 - x^2 ( y^3 + y^5 )  - y^8 $.
	Its Newton polyhedron has three vertices $ (4,0) $, $ (2,3) $, $ (0,8) $.
	Note that $ (2,5) = (2,0) + (0,5) $ is a sum of two vertices of $ \NP(f) $ and $ \NP(g) $, 
	but it is not a vertex of $ \NP(fg) $.
	The normal vectors of the facets of $ \NP(fg) $ are $ e_1, e_2, v_3 := (3,2)^T, v_4 := (5,2)^T $,
	which is precisely the union of the sets of normal vectors of the facets of $ \NP(f) $ and $ \NP(g) $.
	\[
	\begin{tikzpicture}[scale=0.5]
	
	\draw[<->] (0,5)--(0,0)--(9,0);
	
	\path[pattern=north west lines, pattern color=black!20!white,dashed] 
	(0,4.8)--(0,4)--(3,2)--(8,0) -- (8.8,0)-- (8.8,4.8) -- (0,4.8);
	
	\node at (0,5.4) {$ x $};
	\node at (9.4,0) {$ y $};
	
	\foreach \x in {1,...,8}
	\draw (\x,0.1) -- (\x,-0.1) node [below] {\x};
	
	\foreach \x in {1,...,4}
	\draw (0.1,\x) -- (-0.1,\x) node [left] {\x};
	
	\draw[very thick] (0,4.8)--(0,4)--(3,2)--(8,0) -- (8.8,0);

	\filldraw[black] (0,4) circle (2.5pt);
	\filldraw[black] (3,2) circle (2.5pt);
	\filldraw[black] (8,0) circle (2.5pt);
	\end{tikzpicture}
	\hspace{35pt} 
	\begin{tikzpicture}[scale=3]

	\path[pattern=north east lines, pattern color=black!20!white,dashed] 
	(3/2,0)--(0,0) -- (3/2,3/5);
	
	\path[pattern=north west lines, pattern color=black!20!white,dashed] 
	(0,1)--(0,0) -- (3/2,1);
	
	\path[pattern=dots, pattern color=black!20!white,dashed] 
	(3/2,3/5)--(0,0) -- (3/2,1);
	
	\draw[<->] (3/2,0)--(0,0)--(0,1);
	\draw[<->] (3/2,1)--(0,0)--(3/2,3/5);
	
	\node[right] at (1.5,0) {$e_1$};
	\node[left] at (0,1) {$e_2$};
	\node[right] at (3/2,1) {$v_3$};
	\node[right] at (3/2,3/5) {$v_4$};

	\end{tikzpicture} 
	\]  
	\begin{center}
		\small
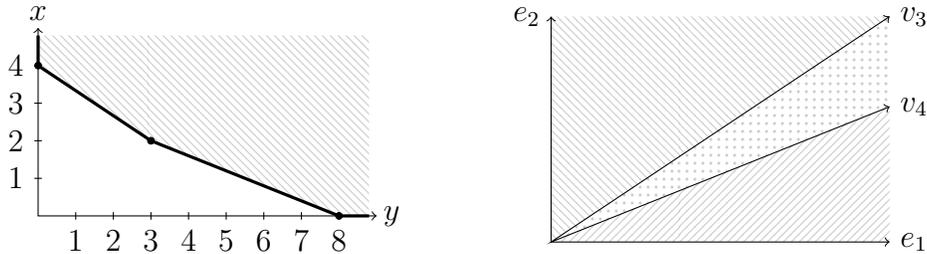
		Newton polyhedron (left) and its dual fan (right).
		
		On the right hand side, the rays are marked with their corresponding generating vectors $ e_1, e_2, v_3, v_4 $.
	\end{center}

	The resulting matrix is 
	$
		M = \begin{pmatrix}
		1 & 0 & 3 & 5\\
		0 & 1 & 2 & 2 \\
		\end{pmatrix},
	$
	which provides the transformation $ \rho \colon \IA^4 \to \IA^2 $
	determined by 
	$ x = x' z_1^3 z_2^5 $ and $ y = y' z_1^2 z_2^2 $.
	Thus $ f $ and $ g $ become
	\[
		\begin{array}{l} 
		f =  x'^2 z_1^6 z_2^{10} - y'^3 z_1^6 z_2^6 
		=  z_1^6 z_2^6  (  x'^2 z_2^{4} - y'^3 )  ,
		\\[5pt]
		g = x'^2 z_1^6 z_2^{10} - y'^4 z_1^{10} z_2^{10} 
		=  z_1^6 z_2^{10} ( x'^2 - y'^4 z_1^{4} )  
		\end{array} 
	\]
	By \eqref{eq:cover}, $ \cF_{\Sigma, \beta} $,
	which is determined by the dual fan of $ \NP(fg) $ and the matrix $ M $,
	can be covered by three charts:
	$ D_+ (y' z_1) $, $ D_+ (x' y') $, and $ D_+ (x' z_2) $.
	We observe that in each of the charts, $ V(f) $ and $ V(g) $ are simultaneously resolved.
	
	\medskip
	
	\item
	Let $ f = x^2 - y^3 $ and $ g = x^4 - z^5 $.
	The picture of the Newton polyhedron  $ \NP (fg; x,y,z) $ is as follows:
	
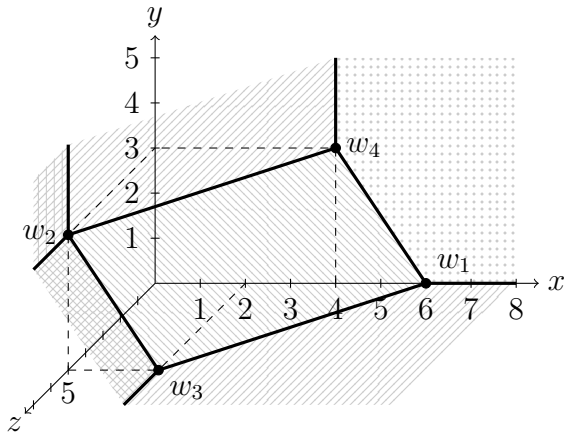
\[
	\begin{tikzpicture}[scale=0.6]

	\path[pattern=dots, pattern color=black!20!white, dashed] 
	(6,0,0) -- (8,0,0) -- (8,5,0) -- (4,5,0) -- (4,3,0) -- (6,0,0);

	\path[pattern=north east lines, pattern color=black!20!white, dashed] 
	(6,0,0) -- (8,0,0) -- (8,0,7) -- (2,0,7) -- (2,0,5) -- (6,0,0);

	\path[pattern=north west lines, pattern color=black!20!white, dashed] 
	(6,0,0)--(2,0,5)--(0,3,5)--(4,3,0) -- (6,0,0);
	
	\path[pattern=north east lines, pattern color=black!20!white, dashed] 
	(0,3,5)--(4,3,0) -- (4,5,0) -- (0,5,5) --(0,3,5);
	
	\path[pattern=north west lines, pattern color=black!20!white, dashed] 
	(2,0,5)--(0,3,5)--(0,3,7) -- (2,0,7)-- (2,0,5);
	\path[pattern=north east lines, pattern color=black!20!white, dashed] 
	(2,0,5)--(0,3,5)--(0,3,7) -- (2,0,7)-- (2,0,5);
	
	\path[pattern=vertical lines, pattern color=black!20!white,dashed] 
	(0,3,5)--(0,3,7) -- (0,5,7) -- (0,5,5) -- (0,3,5);
	\path[pattern=north east lines, pattern color=black!20!white,dashed] 
	(0,3,5)--(0,3,7) -- (0,5,7) -- (0,5,5) -- (0,3,5);

	\draw[->] (0,0,0)--(8.5,0,0);
	\node at (8.9,0,0) {$ x $};
	
	\draw[<->] (0,5.5,0)-- (0,0,0)--(0,0,7.5);
	\node at (0,5.9,0) {$ y $};
	\node at (0,0,8.1) {$ z $};

	\foreach \x in {1,...,8}
	\draw (\x,0.1,0) -- (\x,-0.1,0) node [below] {\x};
	
	\foreach \x in {1,...,7}
	\draw (0,0.1,\x) -- (0,-0.1,\x); 
	
	\node at (0,-0.5,5) {5};
	
	\foreach \x in {1,...,5}
	\draw (0.1,\x,0) -- (-0.1,\x,0) node [left] {\x};
	
	\draw[very thick] (6,0,0)--(2,0,5)--(0,3,5)--(4,3,0) -- (6,0,0);
	
	\draw[very thick] (6,0,0) -- (8,0,0);
	;
	\draw[very thick] (0,5,5) -- (0,3,5) --  (0,3,7);
	\draw[very thick] (2,0,5) -- (2,0,7);
	\draw[very thick] (4,3,0) -- (4,5,0);
	
	\draw[dashed] (0,0,5)--(0,3,5)--(0,3,0);
	\draw[dashed] (0,0,5)--(2,0,5)--(2,0,0);
	\draw[dashed] (4,0,0)--(4,3,0)--(0,3,0);
	
	\filldraw[black] (6,0,0) circle (3pt) node[above right] {$ w_1 $};
	\filldraw[black] (0,3,5) circle (3pt) node[left] {$ w_2 $};
	\filldraw[black] (2,0,5) circle (3pt) node[below right] {$ w_3 $};
	\filldraw[black] (4,3,0) circle (3pt) node[right] {$ w_4 $};
	\end{tikzpicture} 
	\] 
	
	Using {\tt polymake}, we may determine all relevant data of $ \NP (fg) $.
	The vertices are: $ w_1 := (6,0,0), w_2 := (0,3,5), w_3 := (2,0,5), w_4  := (4,3,0) $.
	The rays of the dual fan are
	$ e_1, e_2 , e_3, v_1 := (3,2,0)^T, v_2 := (5,0,4)^T, v_3 := (15,10,12)^T  $.
	Observe that the vector $ v_3 $,
	which corresponds to the compact facet of $ \NP(fg) $,
	is neither a normal vector of $ \NP(f) $ nor of $ \NP(g) $.
	The resulting matrix is
	\[
		M = \begin{pmatrix}
		1 & 0 & 0 & 3 & 5  & 15\\
		0 & 1 & 0 & 2 & 0  & 10\\
		0 & 0 & 1 & 0 & 4  & 12 
		\end{pmatrix}
	\]
	and this determines the transformation
	$
		\IA^6 \to \IA^3
	$
	given by 
	\[
	\begin{array}{ccc} 
		x = x' z_1^3 z_2^5  z_3^{15},
		& 
		y = y' z_1^2 z_3^{10}  , 
		&
		z = z' z_2^4  z_3^{12}. 
	\end{array} 
	\]
	We obtain
	\[
		\begin{array}{l}
			f = x'^2 z_1^6 z_2^{10} z_3^{30}  - y'^3 z_1^6 z_3^{30} 
			= z_1^6 z_3^{30}  ( x'^2 z_2^{10} - y'^3 ), 
			\\[5pt]
			g = x'^4 z_1^{12} z_2^{20} z_3^{60}  - z'^5 z_2^{20} z_3^{60} 
			= z_2^{20}  z_3^{60} ( x'^4 z_1^{12} - z'^5 ) .
		\end{array}
	\]	
	Furthermore, we have (again using {\tt polymake})
	\[
	\begin{array}{|c|c|}
	\hline 
	\mbox{normal vector of facet}
	&
	\mbox{vertices of the facet} 
	\\[5pt]
	\hline 
	\hline 
	e_1 
	&
	w_2 
	\\[5pt]
	\hline 
	e_2 
	&
	w_1, w_3 
	\\[5pt]
	\hline 
	e_3 
	&
	w_1, w_4 
	\\[5pt]
	\hline 
	v_1 
	&
	w_2, w_3 
	\\[5pt]
	\hline 
	v_2 
	& 
	w_2 , w_4 
	\\[5pt]
	\hline 
	v_3 
	&
	w_1, w_2, w_3, w_4
	\\[5pt]
	\hline 
	\end{array}
	\]		
	and from this we determine the rays of the maximal cones of the dual fan:
	(Recall that the maximal cones correspond to the vertices of the Newton polyhedron.)
	\[
	\begin{array}{|c|c|c|}
	\hline 
	\mbox{vertex} & \mbox{rays of corresponding maximal cone} & \mbox{corresponding chart} 
	\\[5pt]
	\hline \hline
	w_1 
	& 
	e_2, e_3, v_3  
	&
	x' z_1 z_2 \neq 0 
	\\[5pt]
	\hline 
	w_2
	&
	e_1, v_1, v_2, v_3 
	&
	y' z' \neq 0  
	\\[5pt]
	\hline 
	w_3
	&
	e_2, v_1, v_3 
	&
	x' z' z_2 \neq 0 
	\\[5pt]
	\hline 
	w_4 
	&
	e_3, v_2, v_3
	&
	x' y' z_1 \neq 0 
	\\[5pt]
	\hline 
	\end{array}
	\]
	Therefore, the corresponding fantastack $ \cF_{\Sigma, \beta} $ can be covered by four charts
	and 
	as one can verify $ V(f) =  V(z_1^6 z_3^{30}  ( x'^2 z_2^{10} - y'^3 )) $ 
	and $ V(g) =  V(z_2^{20}  z_3^{60} ( x'^4 z_1^{12} - z'^5 )) $ are simultaneously resolved in each of them.
	
	\medskip

		\item 
		Let $ f = x^2 -y ^3 z^5 $ and 
		$ g = x^4 - z w^3 $.
		Using {\tt polymake}, we obtain the matrix
		\[
		M(fg) := 
		\left( 
		\begin{array}{cccc|cc|cc|cc}
		1 & 0 & 0 & 0 & 3 & 5 & 1 & 3 & 3 & 5 \\
		0 & 1 & 0 & 0 & 2 & 0 & 0 & 0 & 2 & 0 \\
		0 & 0 & 1 & 0 & 0 & 2 & 4 & 0 & 0 & 2 \\
		0 & 0 & 0 & 1 & 0 & 0 & 0 & 4 & 4 & 6 
		\end{array}
		\right) 
		\]
		(The reason for the separating lines between the columns will become clear soon.)
		Let us study how to it can be constructed by hand from the data of a single binomial.		 
		If we consider only one of them, we get the matrices
		\[
		M(f) :=
		\left( 
		\begin{array}{cccc|cc}
		1 & 0 & 0 & 0 & 3 & 5 \\
		0 & 1 & 0 & 0 & 2 & 0 \\
		0 & 0 & 1 & 0 & 0 & 2 \\
		0 & 0 & 0 & 1 & 0 & 0
		\end{array}
		\right) 
		\ \ \ \ \ \ \
		M(g) :=
		\left( 
		\begin{array}{cccc|cc}
		1 & 0 & 0 & 0 & 1 & 3 \\
		0 & 1 & 0 & 0 & 0 & 0 \\
		0 & 0 & 1 & 0 & 4 & 0 \\
		0 & 0 & 0 & 1 & 0 & 4
		\end{array}
		\right) .
		\]
		The rays which have to be added for the Newton polyhedron of the product $ fg $ arise from
		the intersection of the $ 3 $-dimensional cones 
		of the dual fans, which intersect the relative interior of $ \IR_\gqz^4 $.
		Since we are working with binomials, 
		there is only one such cone for each binomial. 
		More precisely, the rays of the respective cones are
		\[
		(u_1,u_2,u_3)
		:=
		\left( 			\begin{pmatrix}
		3 \\ 2 \\ 0 \\ 0
		\end{pmatrix}
		,
		\begin{pmatrix}
		5 \\ 0 \\ 2 \\ 0
		\end{pmatrix}
		,
		\begin{pmatrix}
		0 \\ 0 \\ 0 \\ 1
		\end{pmatrix}
		\right)
		\ \ \
		(t_1, t_2, t_3) :=
		\left( 			
		\begin{pmatrix}
		1 \\ 0 \\ 4 \\ 0
		\end{pmatrix}
		,
		\begin{pmatrix}
		3 \\ 0 \\ 0 \\ 4
		\end{pmatrix}
		,
		\begin{pmatrix}
		0 \\ 1 \\ 0 \\ 0
		\end{pmatrix}
		\right) 
		\]
		(Take the vectors different from the unit vectors and if the corresponding matrix has a zero row, add unit vectors to the set so that there are no zero rows afterwards.)
		
		A generator of a new ray, which we are looking for, 
		is a vector $ r $, 
		for which there are {\em non-negative integers} $ a,b,c,d,e,f $ such that 
		\[
		r = a u_1 + b u_2 + c u_3 
		= d t_1 + e t_2 + f t_3. 
		\]
		It is tempting to determine a basis of 
		$
		\ker
		\begin{pmatrix}
		u_1 & u_2 & u_3 & -t_1 & -t_2 & -t_3
		\end{pmatrix}
		$,
		but this does not provide the correct vectors since the kernel has dimension two and there are non-trivial choices for the basis. 
		Thus, we have to consider the kernels 
		$ \ker
		\begin{pmatrix}
		u_i & u_j & -t_k & -t_\ell		\end{pmatrix}
		$,
		for varying $ i, j, k, \ell $.
		Since the dimension of this kernel is one,
		the basis is unique up to multiplication by a non-zero constant. 
		(Note that this becomes more delicate if we work with three binomials or more, as we have to repeat this process taking the newly added rays into account.)
		
		Some of the choices for $ i, j ,k, \ell $ cannot provide new rays.
		For example, in 
		\[
		\ker 
		\begin{pmatrix}
		u_2 & u_3 & -t_2 & -t_3
		\end{pmatrix}
		=
		\ker 
		\begin{pmatrix}
		5 & 0 & 3 & 0 \\
		0& 0 & 0 & 1 \\
		2 & 0 & 0 & 0 \\
		0& 1 & 4 & 0 
		\end{pmatrix},
		\]
		the third row is telling us that $ a = 0 $
		and thus, the generator for a potential new ray can be written as $ r = a u_1 + b u_2 + c u_3 =
		c u_3 $,
		which provides the ray generated by $ u_3 $. 
		A computation shows that for the remaining cases, we have:
		\[
		\begin{array}{lcl} 
		\ker
		\begin{pmatrix}
		u_1 & u_3 & -t_2 & -t_3
		\end{pmatrix}
		& = & 
		\mbox{Span}
		(
		(1,4,1,2)^T
		),
		\\[5pt]
		\ker
		\begin{pmatrix}
		u_2 & u_3 & -t_1 & -t_2
		\end{pmatrix}
		& = &
		\mbox{Span}
		(
		(2,12,1,3)^T
		),
		\\[5pt]
		\ker
		\begin{pmatrix}
		u_1 & u_2 & -t_1 & -t_3
		\end{pmatrix}
		& = &
		\mbox{Span}
		(
		(-3,2,1,-6)^T
		).
		\end{array} 	
		\]
		The first two kernels provide the additional rays 
		generated by 
		$ r_1 := u_1 + 4 u_3 = t_2 + 2 t_3 =
		(3,2,0,4)^T $
		and 
		$ r_2 := 
		2 u_2 + 12 u_3 = t_1 + 3 t_2 =
		(10,0,4,12)^T = 2 (5,0,2,6)^T $
		respectively. 
		On the other hand, the third vector $ (-3,2,1,-6)^T $ 
		does not provide a new ray since the signs of the entries varies. 
		In conclusion, we obtain the matrix $ M(fg) $.
		This situation can also be seen in earlier work:  the transverse slice of the dual fan is the linear subspace $ \ell $ introduced before \cite[Lemma~5, p.~1830]{GP} and before \cite[Theorem~3.1]{GPT},  or $ H_\ell $ introduced in \cite[Section 6.1, p.~77]{Teissier}.
		\[
		\begin{tikzpicture}[scale=1]

		\draw[thick] (0,0,4) -- (0,4,0) -- (4,0,0) -- (0,0,4);
		\draw[thick, dashed] (0,0,4)--(0,0,0)--(4,0,0);
		\draw[thick, dashed] (0,0,0)--(0,4,0);
		
		\draw[thick,blue] (12/5,8/5,0) -- (20/7,0,8/7);
		\draw[thick,blue,dashed] (12/5,8/5,0) -- (0,0,0) -- (20/7,0,8/7);

		\draw[thick,red] (4/5,0,16/5) -- (0,4,0);
		\draw[thick,red,dashed] (4/5,0,16/5) -- (12/7,0,0) -- (0,4,0);

		\filldraw[black] (4,0,0) circle (1.75pt) node[right] {$ e_1 $};
		\filldraw[black] (0,4,0) circle (1.75pt) node[left] {$ e_2 $};
		\filldraw[black] (0,0,4) circle (1.75pt) node[left] {$ e_3 $};
		\filldraw[black] (0,0,0) circle (1.75pt) node[below] {\, $ e_4 $};
		
		\filldraw[black] (12/5,8/5,0) circle (1.75pt) node[right] {$ u_1 $};
		\filldraw[black] (20/7,0,8/7) circle (1.75pt) node[below] {$ u_2 $};

		\filldraw[black] (4/5,0,16/5) circle (1.75pt) node[below] {$ t_1 $};
		\filldraw[black] (12/7,0,0) circle (1.75pt) node[above right] {$ t_2 $};

		\filldraw[purple] (12/9,8/9,0) circle (1.75pt) node[left] {$ r_1 $};
		\filldraw[purple] (20/13,0,8/13) circle (1.75pt) node[below] {$ r_2 $};
		
		\draw[thick,purple,dashed] (12/9,8/9,0) -- (20/13,0,8/13);

		\end{tikzpicture} 
		\]
	\end{enumerate}
\end{Ex}

\medskip

In general, 
if we consider more than two binomials, then each binomial is resolved after applying the fantastack construction, but they do not necessarily form a normal crossings divisor.  
The reason for this is that we may have singularities in the torus.
Let us illustrate this for simple examples.

\begin{Ex}
	\phantomsection
	\label{Ex:not_all}
	\begin{enumerate}
	\item 
	If we have $ f_1 := x - 1, f_2 := y - 1, f_3 := xy - 1 $,
	then the fantastack construction does not provide new information since the Newton polyhedron of $ f_1 f_2 f_3 $ is the whole orthant $ \IR_\gqz^2 $. 
	On the other hand, $ V( f_1 f_2 f_3 ) $ does not define a simple normal crossing divisor.
	This can be seen by introducing new coordinates $ \widetilde x := x-1 $, $ \widetilde y := y- 1 $,
	which provides 
	$
	f_1 = \widetilde x,
	f_2 = \widetilde y,
	f_3 = \widetilde x + \widetilde y + \widetilde x \widetilde y $. 
	By blowing up the origin, we will simultaneously resolve $ V(f_1), V(f_2), V(f_3) $.
	
	\item
	Let $ K $ be an algebraically closed field of characteristic $ p > 0 $.
	Consider the binomials $ f_1 := x -1 , f_2 := y-1 $, $ f_3 := xy^p - 1 $.  
	Analogous to (1), 
	we introduce the new variables 
	$ \widetilde x := x-1 $, $ \widetilde y := y- 1 $
	and obtain 
	$ f_1 = \widetilde x $,
	$ f_2 = \widetilde y $,
	$ f_3 = \widetilde x + \widetilde y^p + \widetilde x \widetilde y^p $. 
	Notice that the scheme theoretic intersection of $ V (f_1) $ and $ V (f_3) $ is not reduced.
	In particular, the intersection lattice of $ V(f_1 f_2 f_3) $ does not provide a simple arrangement.
\end{enumerate} 
\end{Ex}

\medskip

Let us restate and prove Theorem~\ref{Thm:Main}.
It generalizes 
Proposition~\ref{Prop:red} and is the best result that could be obtained towards simultaneous normal crossings resolution of finitely many binomial varieties using fantastacks.

\begin{Thm}
	\label{Thm:1Text}
	Let $ K $ be an algebraically closed field of arbitrary characteristic $ p \geq 0 $. 
	Let $ f_1, \ldots, f_a \in K[x] = K [x_1, \ldots, x_m] $ be finitely many binomials,
	where $ a, m  \in \IZ_+ $ with $ m \geq 2 $.
	Let $ \mathcal{E} \subset \IZ $ be {the} set of problematic primes associated to the exponents of the pure binomial factors of $f_1, \ldots, f_a $.

	Let $ \Sigma $ be a subdivision of the dual fan of the Newton polyhedron of the product $ f_1 \cdots f_a $
	and let $ \beta \colon \IZ^n \to \IZ^m $ be the homomorphism 
	determined by the matrix $ M \in \IZ^{m \times n} $,
	whose columns are
	the primitive generators for the rays of the fan $ \Sigma $.
	Let $ \rho \colon  \cF_{\Sigma,\beta} \to \IA^m $ be the morphism of toric stacks induced by $ \beta $.
	Set $ X := V(f_1 \cdots f_a ) \subset \IA^m $.

	Then the reduced preimage $ \rho^{-1}(X)_{\rm red} \subset \cF_{\Sigma, \beta} $ is a sch\"on arrangement of smooth binomials.
	If furthermore $ p \notin \mathcal{E} $,
	then
	$ \rho^{-1}(X)_{\rm red}  $ induces a simple arrangement on $  \cF_{\Sigma, \beta} $.%
\end{Thm}

\begin{proof}
	Set $ X_i := V(f_i) $, for $ i \in \{ 1, \ldots, a \} $.
	Since $ \Sigma $ is a refinement of the normal fan $  \Sigma_{f_i} $. Proposition~\ref{Prop:red} implies that $ \rho^{-1} (X_i)_{\rm red} $ is a sch\"on binomial hypersurface on $  \cF_{\Sigma, \beta} $. 
	Let $ ( y ) = ( y_1, \ldots, y_n ) $, $ n \geq m $, be coordinates in a chart of $ \cF_{\Sigma,\beta} $.
	By the sch\"on  condition, 
	each $ f_i $ is of the form 
	\[ 
		y^{A(i)} ( 1 + \lambda_i y^{B(i)}),
	\]  
	for some units $ \lambda_i \neq 0 $ and $ A(i),B(i) \in \IZ_\gqz^n $. 
	Thus, the only part not being simple normal crossing is the product $ \prod_{i=1}^a ( 1 + \lambda_i y^{B(i)}) $,
	which provides a sch\"on arrangement of smooth binomials.
\end{proof}

\begin{Rk}
	\label{Rk:NumbCharts}
	In the situation of the theorem, 
	we have that $ \cF_{\Sigma,\beta} $ 
	is covered by at most $ 2^a $ affine charts since the number of vertices of $ \NP(f_1 \cdots f_a) $ 
	(and thus the number of maximal cones in $ \Sigma $) 
	is $ \leq 2^a $
	by Observation~\ref{Obs:Simul}(3).
	
	In general, $ 2^a $ is strictly smaller than the number of affine charts obtained by blow-ups in smooth centers. 
	For example, a possible desingularization of $ x_1 x_2 - x_3 x_4 $ via blow-ups in smooth centers could be to choose the closed point as the center and hence creates $ 4 $ charts. 
	This difference becomes even bigger if we consider binomials with large exponents. 
	For further explicit examples, we refer to~\cite[Examples~8.4 -- 8.11]{Gaube},
	where different methods for the local monomialization of a single binomial are studied, implemented, and compared.
	A local monomialization is a local variant of desingularization,
	where one does not necessarily demand the centers to be chosen globally.
	
\end{Rk}

If $ \mbox{char}(K) \notin \mathcal{E} $, then the remaining simultaneous desingularization is obtained 
by applying the following theorem.
A more general variant of this result was proven in \cite[Theorem~1.3]{Li}.

\begin{Thm}[{\cite[Theorem~1.1]{Hu}}]
	\label{Thm:Hu}
	Let $ {X_0} $ be an open subset of a non-singular algebraic variety $ X $
	such that $ X \setminus {X_0} $ can be decomposed as a union $ \bigcup_{i \in I} S_i $ 
	of closed irreducible subvarieties with the properties
	\begin{enumerate}
		\item[$(i)$] $ S_i $ is smooth,
		
		\item[$(ii)$] $ S_i $ and $ S_j $ meet cleanly, and
		
		\item[$(iii)$] $ S_i \cap S_j = \varnothing $ or a disjoint union of $ S_\ell $. 
	\end{enumerate}
	The set $ \cS  = \{ S_i \}_i $
	is a poset.
	Let $ k $ be the rank of $ \cS $.
	Then there exists a sequence of well-defined blow-ups
	\[
		\Bl_\cS (X) \to \Bl_{\cS_{\leq k -1}}(X) \to \ldots \to \Bl_{\cS_{\leq 0}} (X) \to X,
	\]
	where $ \Bl_{\cS_{\leq 0}} (X) \to X $ is the blow-up of $ X $ with center all $ S_i $ of rank $ 0 $
	and
	$ \Bl_{\cS_{\leq r}} (X) \to \Bl_{\cS_{\leq r-1}} (X) $ is the blow-up of $ \Bl_{\cS_{\leq r-1}} (X) $ with center the proper transform of all $ S_j $ of rank $ r $, 
	for $ r \geq 1 $. 
	We have 
	\begin{enumerate}
		\item $ \Bl_\cS (X) $ is smooth,
		
		\item $ \Bl_\cS (X) \setminus {X_0} = \bigcup_{i \in I } \widetilde{S}_i $ is a normal crossings divisor,
		and 
		
		\item 
		$ \widetilde{S}_{i_1} \cap \ldots \cap \widetilde{S}_{i_n} $
		is non-empty if and only if $ S_{i_1}, \ldots, S_{i_n} $ form a chain the the poset $ \cS $. 
		Hence, $ \widetilde{S}_i $ and $ \widetilde{S}_j $ meet if and only if $ S_i $ and $ S_j $ are comparable. 
	\end{enumerate}
\end{Thm}

As explained in \cite[Introduction/after Theorem 1.2]{Li}, this is a generalization of De Concini and Procesi's wonderful models of subspace arrangements \cite{DCP}.

We need a variant of the theorem for algebraic stacks. We note the following direct consequence of smoothness:

\begin{Prop} 
Let $X_0 \subset X$ be as in the theorem, and $\pi:Y \to X$ a smooth morphism.
Set  
$ Y_0 := \pi^{-1} (X_0) $, $S_i^Y  := \pi^{-1} (S_i)$ and $\cS^Y = \{S_i^Y\}$. 
Then $S_i^Y$ satisfy properties (i)-(iii), and, setting  $\Bl_{\cS^Y_{\leq r}}(Y)$ as in the theorem, we have that
 $$\Bl_{\cS^Y_{\leq {r}}}(Y)  
 = Y \times_X\Bl_{\cS_{\leq r}}(X),$$ and the sequence
	\[
		\Bl_{\cS^Y} (Y) \to \Bl_{\cS^Y_{\leq k -1}}(Y) \to \ldots \to \Bl_{\cS^Y_{\leq 0}} (Y) \to Y
	\]
satisfies the conclusions (1)-(3) of the theorem.
\end{Prop}

Given an algebraic stack $\cX$, an open $\cX_0\subset \cX $ with complement $ \bigcup_i S_i^\cX$ satisfying $(i)$--$(iii)$ of the theorem, and any smooth presentation by schemes $Y_1 \double Y_0$, we obtain pullback subschemes $S_i^{Y_1}$ and $S_i^{Y_0}$ such that $S_i^{Y_1}$ are the preimage of $S_i^{Y_0}$ under either projection $Y_1 \to Y_0$. It follows that we obtain smooth groupoids $$\Bl_{\cS^{Y_1}_{\leq r}}(Y_1) \double \Bl_{\cS^{Y_0}_{\leq r}}(Y_0)$$ with quotient stacks $\Bl_{\cS^{\cX}_{\leq r}}(\cX)$ satisfying the conclusion of the theorem. We therefore obtain:
  
  \begin{Cor} The theorem applies as stated to algebraic stacks.
  \end{Cor}

\subsection{What happens in mixed characteristics}
\label{subsec:mixed}
While the theory of toric Artin stacks is only established over algebraically closed fields, we are interested in applications such as those discussed in Section~\ref{Sec:p-adic-intro}, and the methods provided here can be applied and sometimes work. We discuss here what works without change and what  requires further work.

\smallskip 

We consider a collection of binomials $f_1,\ldots,f_a$ over the ring of integers $\cO_K$ of a finite extension $K$ of $\IQ_p$. 

First, the formation of $\rho: \cF_{\Sigma, \beta} \to \IA^m$ depends only on the combinatorial data of the exponents appearing in $f_i$, and works without change over an arbitrary base, but with weaker outcomes.

Second, the treatment of singularities here requires working with binomials of the form $x^A - \lambda x^B$ where \emph{$\lambda$ is a unit}.  
Without this assumption the procedure still leads to equations of the form  $x^C(x^A - \lambda)$ (or $x^C(1 - \lambda x^B)$), but if $v(\lambda)>0$ the first of these is no longer sch\"on. We expect that Kato's theory of toric singularities in mixed characteristics \cite{Kato-toric} will provide a way to treat these cases; 
another approach, suggested by the referee, is to systematically work with toric schemes over valuation rings --- this is consistent with the work \cite{DM_p-adic}. This will have to be addressed elsewhere.

Third, when $f = x^A - \lambda y^B$ is a pure binomial and the residue characteristic $p$ divides all the exponents, the hypersurface $f= 0$ in the torus is not smooth over $\cO_K$. Over a perfect field its reduction is smooth, and in general its reduction is regular, being the orbit of a smooth group scheme.  

Fourth, even when $p$ does not divide the exponents of individual binomials, if $p\in \cE$ the binomials do not meet cleanly. As stated earlier, this situation is interesting already over an algebraically closed field of characteristic $p$.

Finally, if all $\lambda_i$ are units and $p\notin \cE$ then the procedure does provide an arrangement of sch\"on binomials meeting cleanly, in which case the procedure for resolving the arrangement in Theorem \ref{Thm:Hu} works without change.

\section{Binomial ideals}

\label{binomial ideal}
The purpose of this section is to generalize  Proposition~\ref{Prop:red} 
to resolution of singularities for binomial {subschemes} $ X \subset \IA^m $ (Theorem~\ref{Thm:MainIdeals}).
As we neglect the question of functoriality, 
we can work with a fixed system of generators for the binomial ideal defining $ X $.

Note that the situation is a bit simpler 
than in the simultaneous setup in the previous section.
For example, in both cases of Example~\ref{Ex:not_all}, the ideal generated by $ f_1, f_2, f_3 $ is equal to $ \langle x-1 , y-1 \rangle $
and
hence, the corresponding variety is smooth.

\begin{Thm}[Theorem~\ref{Thm:MainIdeals}]
	\label{Thm:2Text}

		Let $ K $ be an algebraically closed field and $ m \in \IZ_{\geq 2} $.
		Let $ X \subset \IA^m $ be a binomial 		subscheme.
		Denote by $ \Sigma_0 $  the fan in $ \IR^m $ with unique maximal cone $ \IR_\gqz^m $. Let $f = \{f_1,\ldots,f_a\}$ be binomial generators of $\cI_X$. 
		\begin{enumerate}
		 \item {\rm (See \cite{GP, GPT, Teissier})} 
		 The proper transform of $X^\pure$ in $\cF_{\Sigma_f,\beta}$ is a sch\"on purely binomial subscheme, in particular its reduction is smooth.
		 \item 
		There exists a further subdivision $ \Sigma $ of $ \Sigma_0 $, of combinatorially bounded complexity,
		with the following property:
		If $ \beta \colon \IZ^n \to \IZ^m $ is the homomorphism 
		determined by the matrix $ M \in \IZ^{m \times n} $,
		whose columns are
		the primitive generators for the rays of the fan $ \Sigma $,
		then
		\[
				\rho^{-1}(X)_{\rm red}  \subset  \cF_{\Sigma,\beta}
		\]
		is in simple normal position on $ \cF_{\Sigma,\beta} $ (see Section \ref{Sec:mainresult}),
		where we denote by $ \rho \colon  \cF_{\Sigma,\beta} \to \IA^m $ the morphism of toric stacks induced by $ \beta $.
		\end{enumerate}

\end{Thm}

\begin{proof}
	Without loss of generality, we assume that,
	for every $ i \in \{ 1, \ldots, a \} $,
	$ f_i $ is not contained in the ideal generated by $ \{ f_1, \ldots, f_a \}  \setminus \{ f_i \} $.
	Note that some $ f_i $ might already be monomials. 
	
	As before, $ \Sigma_f = \Sigma_{f_1\cdots f_a} $ is the dual fan of the Newton polyhedron of the product $ f_1 \cdots f_a $ and let $ \rho_{\Sigma_f} \colon \cF_{\Sigma_f, \beta} \to \IA^m $ 
	be the morphism of toric stacks as in Theorem~\ref{Thm:Main}.
	In particular, in every chart, $ \rho_{\Sigma_f}^{-1} (f_i) $ is of the form $ y^{C(i)} (1 - \epsilon_i y^{A(i)} ) $ and $ f_i' :=  1 - \epsilon_i y^{A(i)} $ is sch\"on by Theorem~\ref{Thm:Main},
	where $ (y) $ is a system of local coordinates and $ \epsilon_i $ is either a unit or zero. 
	(For this, recall Remark~\ref{Rk:non_red} and use that every binomial is a product of reduced binomials -- possibly with multiplicities -- since $ K $ is algebraically closed). 
	
	We claim that in each chart $X'=V( f_1',\ldots f_a')$ is a sch\"on purely binomial subscheme. Indeed the variables $y_j$ appearing in the $f_i'$ are all invertible on $X'$, and  so $X'= X'_y \times \IA^{m'}$ is the product of a coset of a subgroup of the torus in these parameters with the affine space with the remaining coordinates $x_l$. Its reduction is thus smooth and it meets $V(\prod x_l)$ transversely. This gives part (1).

To address Part (2), we need to modify 
	the monomial factors which on each chart are of the form $ y^{C(1)}, \ldots, y^{C(a)} $. We note that the locally principal  ideals given on charts as $ \langle y^{C(1)}\rangle, \ldots, \langle y^{C(a)}\rangle $ are well-defined globally, since $ \langle f_i\rangle = \langle f_i' \rangle    \langle y^{C(i)} \rangle $.

	We define monomial ideals
	\[  
		G_{i,j} := \left\langle y^{C(i)} , y^{C(j)} \right\rangle ,
	\]  
	for $ i, j \in \{ 1, \ldots, a \} $ with $ i < j $.
	Let $ \Sigma_G $ be the dual fan of the Newton polyhedron of $ \prod_{i<j} G_{i,j} $, or, equivalently, of the product of binomials $ \prod_{i<j} \left(y^{C(i)} - \lambda_{ij} y^{C(j)}\right)$ with $\lambda_{ij}\neq 0,1$. 
	Let $\Sigma$ the induced subdivision of $\Sigma_f$. Equivalently, this is the fan of the Newton polyhedron of  $ \prod_{i<j} G_{i,j} \times \prod_i  \langle  f_i  \rangle  $, or equivalently, of the product of binomials $ \prod_{i<j} \left(y^{C(i)} - \lambda_{ij} y^{C(j)}\right) \times \prod_i f_i$.
	
	Applying the fantastack construction for $ \Sigma$, 
	we achieve that for every subset $ J \subseteq \{ 1, \ldots, a \} $ the ideal $ \langle y^{C(j)} \mid j \in J \rangle $ is principalized in every chart.
	The latter has the consequence that the monomial factors are totally ordered: 
	in each chart, after a possible relabeling of the generators (and abusing notation), we have 
	$ y^{C(i)} \mid y^{C(i+1)}  $ for every $ i \in \{ 1, \ldots, a-1 \} $. 
	
	Let $ M_i $ be the reduced monomial obtain from $ y^{C(i)} $.
	Set $ N_1:= M_1 $ and $ N_i := M_i/M_{i-1} $ for  
	$ i \in \{ 2, \ldots , a \} $. 
	Then the reduced preimage of $ X $ in the given chart decomposes as 
	\[
		V(N_1) \cup V(N_2, f_1') \cup V(N_3, f_1', f_2') \cup \dots\cup V(N_a, f_1',\dots, f_{a-1}') \cup V(f_1',\dots,f_a'),
	\]
	where we write $ V(I) $ for the scheme determined by the vanishing locus of the ideal generated by the elements of $ I $.
	One checkes that this is in normal crossing position, as required.

Notice once again that the number of cones of the fan and therefore the number of charts and divisors, 
	does only depend on the ambient dimension and the number $ a $ of binomial generators, 
but not on the explicit exponents of the binomials.

\end{proof}

	By applying a combinatorial sequence of blowups, as for example in \cite{Li}, 
	we may obtain a divisor with simple normal crossings:
	
	\begin{Prop} \label{Thm:Li}
	There exists a sequence of at most $m$ blowups $L: Y \to \cF_{\Sigma,\beta}$ such that  $L^{-1} \rho^{-1} (X)$ is a simple normal crossings divisor.
	\end{Prop}
	\begin{proof} The closures $S_I$ of  strata  of a scheme in simple normal position form a building set in the sense of \cite[Section 2.1]{Li}. Let $G_c$ be the union of the $S_I$ of codimension $c$. By the theorem this applies to $\rho^{-1}(X)$.  By \cite[Proposition 5.3]{Li} blowing up, in decreasing order of codimension, the successive proper transforms of $G_c$ one obtains the result.
	\end{proof}

\end{document}